%% file: paper.tex
\documentclass[12pt]{article}
\usepackage{amsmath}
\usepackage{amssymb}
\usepackage{stmaryrd}
\usepackage{amsthm}

\usepackage{cancel}

\usepackage[colorlinks=true,citecolor=black,linkcolor=black,urlcolor=blue]{hyperref}

\usepackage[dvipsnames]{pstricks} 
\usepackage{multido}

\usepackage{graphicx}
\usepackage{xcolor}

\newcommand{\showgrid}{}

\newcommand{\gridon}{\renewcommand{\showgrid}{\psset{subgriddiv=1,griddots=10,gridlabels=6pt}\psgrid}}

\gridon 

\newgray{hellgrau}{.89}

\usepackage[applemac]{inputenc}

\usepackage{algorithmic}

\input config

\englishtrue
\input mf_makros
\input mf-theorems

\def\figref#1{Fi\-gu\-re~\ref{#1}}
\def\defeq{:=}
\def\qpoch#1#2#3{\pas{#1;#2}_{#3}}

\newgray{mfgray85}{0.85}

\newgray{mfgray70}{0.70}

\newgray{mfgray60}{0.60}

\title{The quotient of generating functions of lozenge tilings for certain regions derived from hexagons, obtained with non--intersecting lattice paths}

\author{Markus Fulmek\thanks{
Research supported by the National Research Network ``Analytic
Combinatorics and Probabilistic Number Theory'', funded by the
Austrian Science Foundation. 
}\\
\small Fakult\"at f\"ur Mathematik \\
\small Oskar-Morgenstern-Platz 1, A-1090 Wien, Austria \\
\small \tt Markus.Fulmek@Univie.Ac.At
}

\date{2020} 

\def\secA{\section}
\def\secB{\subsection}

\def\EM#1{{\em #1\/}}
\begin{document}
\bibliographystyle{plain}


\long\def\insertfigure#1#2{%
\begin{figure}%
\begin{center}%
\input graphics/#1%
\end{center}%
\caption{#2}%
\label{fig:#1}%
\end{figure}%
}

\long\def\inserttwocommentedfigures#1#2#3#4{%
\begin{figure}%
\begin{center}%
\input graphics/#1%
\hfil%
\input graphics/#2%
\end{center}%
{\small #4}%
\caption{#3}%
\label{fig:#1-#2}%
\end{figure}%
}

\long\def\insertcommentedfigure#1#2#3{%
\begin{figure}%
\begin{center}%
\input graphics/#1%
\end{center}%
{{\small #3}}
\caption{#2}%
\label{fig:#1}%
\end{figure}%
}

\parindent0pt
\parskip1em

\def\gf#1{{\mathfrak{gf}}\!\of{#1}}
\def\gfone#1{{\overline{\mathfrak{gf}}}\!\of{#1}}
\def\gftwo#1{{\widehat{\mathfrak{gf_1}}}\!\of{#1}}
\def\gfthree#1{{\widehat{\mathfrak{gf_0}}}\!\of{#1}}

\def\gfhalfeven#1{{\mathbf{gf}_0}\!\of{#1}}
\def\gfhalfodd#1{{\mathbf{gf}_1}\!\of{#1}}
\def\laiweight#1{\pas{X q^{#1}+Y q^{-\pas{#1}}}}
\def\maindet#1{{\mathbf d}\of{#1}}
\def\mfmat#1{{\mathbf m}\of{#1}}
\def\mfmatprime#1{{\mathbf m^\prime}\of{#1}}
\def\mfmatpprime#1{{\mathbf m^{\prime\prime}}\of{#1}}
\def\mfmatppprime{{\mathbf m^{\prime\prime\prime}}}
\def\mfsimple#1{{\mathbf s}\of{#1}}
\def\mffac#1{{\mathbf f}\of{#1}}
\def\mfdet#1{{\mathbf d}\of{#1}}
\def\minor#1#2#3#4{\brk{\maindet{#1}}_{#2,#3}^{#4}}
\def\mfminor#1#2#3#4{\brk{\mfdet{#1}}_{#2,#3}^{#4}}
\def\laplaceminor#1#2{\brk{\maindet{#1}}_{\cancel 1,\cancel #2}}

\def\mfentry#1{{\mathbf e}\of{#1}}

\def\prodcoeff#1#2#3{\mu^{#1}_{#2,#3}}

\def\mfseq{{\mathbf a}}

\def\qbinom#1#2{\genfrac{[}{]}{0pt}{}{#1}{#2}_q}
\def\qexp#1#2{e\of{#1,#2}}
\def\inner#1#2{\left\langle#1,#2\right\rangle}
\def\ratfield{\F}
\def\pprime{{\prime\prime}}
\def\mainminor#1#2#3{\left[{#1}\right]_{i\geq #2,j\geq #3}}
\def\step#1{\subsubsection*{#1:}}

\maketitle

\begin{abstract}
In a recent preprint, Lai
showed that the quotient of generating functions of weighted lozenge tilings of two ``half hexagons
with lateral dents'', which differ only in width, factors nicely, and the same is true for the
quotient of generating functions of weighted lozenge tilings of  two ``quarter hexagons
with lateral dents''. Lai achieved this by using ``graphical condensation''
(i.e., application of a certain Pfaffian identity to the weighted enumeration of matchings).

The purpose of this note is to exhibit how this can be
done by the Lindström--Gessel--Viennot method for nonintersecting lattice paths.
For the case of ``half hexagons'', basically the same observation, but restricted to mere enumeration (i.e., all
weights of lozenge tilings are equal to $1$), is contained in a recent preprint of Condon.
\end{abstract}

\secA{Lai's observation for lozenge tilings}
In a recent preprint, Lai \cite{Lai:2020:ROTGFOSAQH} considers lozenge tilings of ``half hexagons with lateral dents''
and of ``quarter hexagons with lateral dents'' and shows that the quotient of the generating functions of such tilings
(with a certain weight) for two objects (i.e., two ``half'' or two ``quarter'' hexagons) which differ only in width (to be explained
below) has a nice factorization. We shall show how these observations 
can be
obtained by the 
the Lindström--Gessel--Viennot method \cite{Lindstroem:1973:OTVROIM,Gessel-Viennot:1998:DPAPP} of non--intersecting lattice paths.

\secA{Half hexagons}
We shall start with the case of ``half hexagons'', since it is much simpler. 

The literature on tilings enumerations is abundant (see, for instance, \cite{Ciucu-Eisenkoelbl-Krattenthaler-Zare:2001:EOLT}); for the experienced reader it certainly suffices to have
a look at the left picture in \figref{fig:lai-hex-lai-nlp}: A ``half hexagon'' is simply the upper half of some
hexagon with a horizontal symmetry axis, drawn in the triangular lattice; and ``lateral dents'' are triangles of this
``half hexagon'' adjacent to its lateral sides which were \EM{removed} from the ``half hexagon''. All \EM{vertical}
lozenges of a tiling are labelled: This labelling is \EM{vertically constant} and \EM{horizontally increasing by $1$ from
left to right}, such that all vertical lozenges bisected by the vertical symmetry axis of the ``half hexagon'' have
label $0$ (see the left picture in \figref{fig:lai-hex-lai-nlp}). Let $T$ be some lozenge tiling whose vertical
lozenges are labelled ${v_1, v_2, \dots, v_m}$, then the weight of $T$ is defined as
$$
w\of T\defeq\prod_{i=1}^m \frac{X q^{v_i} + Y q^{-v_i}}2.
$$
Lai observed that if \EM{only the width $x$} (i.e., the length of the upper horizontal side) of such ``half hexagon with lateral dents''
is changed (i.e., the height and the relative positions of the lateral dents are unchanged), then the corresponding generating function of all tilings (weighted as described above) changes by a
simple product which
does not contain the variables $X$ or $Y$. Lai provided a proof for this
fact by ``graphical condensation'' (i.e., application of a certain Pfaffian identity to the enumeration of matchings).

\begin{figure}%
\begin{center}%
\input graphics/lai-hex%
\hfil%
\input graphics/lai-nlp%
\end{center}%
{\small %
The left picture shows a ``half hexagon'' with side lengths $12, 7, 5, 7$ in the triangular lattice: The lateral sides
have ``dents'' (i.e., missing triangles; indicated in the picture by black colour), $4$ on the left side and $3$ on the
right side. The triangle ``on top'' of this ``half hexagon'' shows the labelling of the \EM{vertical lozenges}, which is
\EM{constant} vertically and \EM{increasing by $1$} horizontally (from left to right). The picture also shows a \EM{lozenge
tiling} of this ``half hexagon with dents'', where the three possible orientations of lozenges ({\bf\color{Mahogany} left--tilted},
{\bf\color{Apricot} right--tilted}
and {\bf\color{Tan} vertical}) are indicated by three different colours: This particular tiling has weight
{\footnotesize $$
w_{-7}^2\cdot w_{-6}\cdot w_{-1}\cdot w_{0}\cdot w_{3}\cdot w_{6},
$$}%
where $w_i\defeq\frac{X q^i + Y q^{-i}}2$. %
The non--intersecting lattice paths corresponding to this tiling are indicated by white lines in the
left picture; the right picture shows a ``reflected, rotated and tilted'' version of these paths in the lattice $\Z\times\Z$,
where horizontal edges $\pas{a,b}\to\pas{a+1,b}$ are labelled $a-2b$ (these labels are shown in the right picture
only for the region of interest in our context, i.e., for $0\leq y\leq x$). Clearly, this bijection between lozenge tilings
and non--intersecting lattice paths (introduced here ``graphically'') is \EM{weight--preserving} if we define the
weight of some family $P$ of of non--intersecting lattice paths as the product of $w_i$, where $i$ runs over the
labels of all horizontal edges belonging to paths in $P$. %
}%
\caption{Pictures corresponding to Figures 1.2.a and 2.1.a in Lai's preprint: %
The length of the upper horizontal side of the ``half hexagon'' in the left picture is the ``width parameter'' $x$
considered by Lai (so $x=5$
in this picture; the ``height'' by definition is equal to the length of the lateral sides, so it is $7$
in this picture).}%
\label{fig:lai-hex-lai-nlp}%
\end{figure}%

\secA{Translation to non--intersecting lattice paths}
The literature on the connection between lozenge tilings and non--intersecting lattice paths is abundant
(see, for instance, \cite[Section 5]{Ciucu-Eisenkoelbl-Krattenthaler-Zare:2001:EOLT});
for the experienced reader it certainly suffices to have
a look at the pictures in \figref{fig:lai-hex-lai-nlp}: It is easy to
see that there is a weight--preserving bijection between lozenge tilings and families of non--intersecting lattice paths in the 
lattice $\Z \times \Z$ with steps to the right and downwards, where steps to the right
from $\pas{a,b}$ to $\pas{a+1,b}$ are labelled $a-2b$ and thus have weight 
$$
\frac{X q^{a-2b} + Y q^{2a-b}}{2}
$$
(and all downward steps have weight $1$). As usual, the weight of a lattice path is the product of all
the weights of steps it consists of.

It is easy to see that the generating function of all lattice paths from initial point $\pas{a, b}$ to terminal point $\pas{c, d}$
is zero for $a>c$ or $b<d$, otherwise it is equal to:
\begin{equation}
\label{eq:gf-lattice-path-general}
\gf{a,b,c,d} = \prod_{j=1}^{c-a}\frac{\pas{X q^{j-1-2b+a}+Y q^{-j+1+2d-a}}\pas{1-q^{2\pas{b-d}+2j}}}{2\pas{1-q^{2j}}}.
\end{equation}
This follows immediately by showing that \eqref{eq:gf-lattice-path-general} fulfils the recursion for such
weighted lattice paths
\begin{align*}
\gf{a,b,a,d} &\equiv 1, \\ 
\gf{a,b,c,b} &= \prod_{i=a-2b}^{c-2b-1}\frac{X q^i + Y q ^{-i}}2, \\ 
\gf{a,b,c,d} &= \frac{X q^{a-2b} + Y q ^{2b-a}}2\gf{a+1,b,c,d} + \gf{a,b-1,c,d}
\end{align*}
for $a\leq c$ and $b\geq d$.
 
We have to specialize this to our situation, i.e., to initial point $\pas{a,a}$ and terminal point $\pas{c,0}$
(see the right picture in \figref{fig:lai-hex-lai-nlp}):
The generating function $\gfone{a,c}$ of all lattice paths from $\pas{a, a}$ to $\pas{c,0}$ is zero for $c<a$,
and for $c\geq a$ it is equal to
{\small
\begin{equation}
\label{eq:gf-lattice-path}
\gfone{a,c} = \gf{a,a,c,0} =
\pas{2 q^a}^{\pas{a-c}}
\prod_{j=1}^{c-a}\frac{\pas{X q^{j-1}+Y q^{1-j}}\pas{1-q^{2a+2j}}}{1-q^{2j}}.
\end{equation}
}

Note that increasing the \EM{width} of the ``half hexagon with lateral dents'' by some $d\in\N$ corresponds
bijectively to shifting all initial and terminal points of the corresponding
non--intersecting lattice paths (i.e.,  $\pas{a,a}\to\pas{a+d,a+d}$ and $\pas{c,0}\to\pas{c+d,0}$), and from 
\eqref{eq:gf-lattice-path} we immediately obtain 
$$
\gfone{a+d,c+d} = \gfone{a,c}\cdot 
\prod_{j=1}^{c-a}\frac{{1-q^{2a+2d+2j}}}{q^d\pas{1-q^{2a+2j}}}.
$$
By using the standard $q$--Pochhammer notation $\qpoch aq0 \defeq 1$ and
\begin{align}
\qpoch aqn &\defeq\prod_{j=0}^{n-1}\pas{1-a\cdot q^j}
\label{eq:qpochplus}
\\
\qpoch aq{-n} &\defeq\frac1{\qpoch{a\cdot q^{-1}}{q^{-1}}n}
\label{eq:qpochminus}
\end{align}
for integer $n>0$,
we may rewrite this as follows:
\begin{equation}
\label{eq:gf-lattice-path-factor}
\gfone{a+d,c+d} = \gfone{a,c}\cdot
\frac{\qpoch{q^{2d+2}}{q^2}{c}}{q^{d c}\cdot\qpoch{q^2}{q^2}{c}}
\cdot
\frac{{q^{d a}\cdot\qpoch{q^2}{q^2}{a}}}{\qpoch{q^{2d+2}}{q^2}{a}}.
\end{equation}
(Note that Lai uses a different notation: In \cite[equation (1.1)]{Lai:2020:ROTGFOSAQH},
$q$ is replaced by $-q$.)

By the well--known Lindström--Gessel--Viennot argument \cite{Lindstroem:1973:OTVROIM,Gessel-Viennot:1998:DPAPP}, the generating function of all families
of non--intersecting lattice paths can be written as a determinant, and
by the \EM{multilinearity} of the determinant, we obtain for all $n\in\N$ and all $n$--tuples
$\pas{a_1 < a_2 <\cdots < a_n}$ and $\pas{c_1 < c_2 < \cdots < c_n}$ with
$a_i\leq c_i$, $1\leq i \leq n$:
{\small
\begin{equation}
\label{eq:lais-observation}
\frac{\det\pas{\gfone{a_i+d,c_j+d}}_{i,j=1}^n}{\det\pas{\gfone{a_i,c_j}}_{i,j=1}^n}
=
\prod_{m=1}^n\pas{
\frac{\qpoch{q^{2d+2}}{q^2}{c_m}}{q^{d c_m}\cdot\qpoch{q^2}{q^2}{c_m}}
\cdot
\frac{{q^{d a_m}\cdot\qpoch{q^2}{q^2}{a_m}}}{\qpoch{q^{2d+2}}{q^2}{a_m}}
}.
\end{equation}
}

By the weight--preserving bijection between lozenge tilings and families of non--intersect\-ing lattice paths, this
is equivalent to Lai's observation \cite[Theorem 1.1]{Lai:2020:ROTGFOSAQH}.
(Basically the same simple approach,
but restricted to mere enumeration, is contained
in a recent preprint of Condon \cite{Condon:2020:LTFRFHWDOTS}.)

\secA{Quarter hexagons}
Lai considered other regions, namely ``quarter hexagons'' with dents: 
Again, the idea should become clear
quickly when looking at a picture, see \figref{fig:lai-hex-quarter}. By the bijection introduced ``graphically'' in
\figref{fig:lai-hex-lai-nlp}, such tilings correspond to families of non--intersecting lattice paths with
initial points $\pas{2b,b}$ and terminal points $\pas{c,0}$ (see \figref{fig:lai-nlp-1-5-b}, where the
same tiling as in \figref{fig:lai-hex-quarter} is considered, but with labels of vertical lozenges
all increased by $1$: This ``shift of labels'' implies starting points $\pas{2b+1,b}$
for the lattice paths).

\begin{figure}%
\begin{center}%
\input graphics/lai-hex-quarter%
\end{center}%
{{\small The picture shows a lozenge tiling of
the ``half of a half hexagon with dents'' (width $8$ and height $12$). The particular tiling shown here is
the same as in
\cite[Figure 1.5.b]{Lai:2020:ROTGFOSAQH}.
Note that the \EM{same picture}, but with all
labels of vertical lozenges \EM{increased
by $1$}, also corresponds to the ``half of a half hexagon with dents'' of the same height $12$,
but of width $9$: This would give the tiling in \cite[Figure 1.5.a]{Lai:2020:ROTGFOSAQH}.}}
\caption{A lozenge tiling of a ``quarter hexagon'' (i.e., the ``half of a half hexagon'')
of width $8$ and height $12$.}%
\label{fig:lai-hex-quarter}%
\end{figure}%

\begin{figure}%
\begin{center}%
\input graphics/lai-nlp-1-5-b%
\end{center}%
{{\small The picture shows the family of non--intersecting lattice paths corresponding
bijectively to the tiling in \figref{fig:lai-hex-quarter} but with all labels \EM{increased by $1$}: The bijection
is the same as the one ``graphically'' introduced in \figref{fig:lai-hex-lai-nlp}. All lattice paths start at
some point $\pas{2b+1,b}$ and end in some point $\pas{c,0}$,
where $c\geq 2b+1$. The tiling of \figref{fig:lai-hex-quarter} (with \EM{unchanged} labels)  would
correspond to the same family of lattice paths \EM{shifted to the left by $1$}, i.e., with initial points
$\pas{2b,b}$ and terminal points $\pas{c-1,0}$.}}
\caption{Non--intersecting lattice paths corresponding
to \figref{fig:lai-hex-quarter}.}%
\label{fig:lai-nlp-1-5-b}%
\end{figure}%

For such ``quarter hexagons'', Lai considered basically the same weight as for the ``half hexagons'', but with two
modifications:
\bit
\item We let $X=Y=1$,
\item and we let the weight of lozenges with label $0$ be $\frac12$ (\EM{not} $\frac{q^0+q^{-0}}2=1$).
\eit
Observe that for $X=Y=1$ we may rewrite \eqref{eq:gf-lattice-path-general} as follows (using again
notation \eqref{eq:qpochplus}):
\begin{multline}
\label{eq:gfx=y=1}
\left.\gf{a, b, c, d}\right\vert_{X=Y=1} = \\
2^{a-c}q^{\frac12\pas{a-c}\pas{a+c-1-4d}}
\frac{\qpoch{-q^{2a-2b-2d}}{q^2}{c-a}\qpoch{q^{2b-2d+2}}{q^2}{c-a}}{\qpoch{q^2}{q^2}{c-a}}
\end{multline}
So defining 
\begin{align*}
\gfthree{b,c} &\defeq\frac12\left.\gf{2b+1, b, c,0}\right\vert_{X=Y=1}+\left.\gf{2b, b-1, c,0}\right\vert_{X=Y=1},\\
\gftwo{b,c}  &\defeq\left.\gf{2b+1, b, c,0}\right\vert_{X=Y=1},
\end{align*} we obtain from
\eqref{eq:gfx=y=1}:
\begin{align}
\gfthree{b, c} &=
\frac{2^{2b-c}q^{b\pas{2b-1}-\frac12 \pas{c-1}c}\pas{1-q^{2c}}}{1-q^{4b}}
	{
		\frac{
		\qpoch{q^{4b+4}}{q^4}{c-2b}
	}{
		\qpoch{q^2}{q^2}{c-2b}
	}
	} \text{ for } c\geq2b+1, \label{eq:quarter0} \\
\gftwo{b, c} &= 
	2^{2b+1-c}q^{b\pas{2b+1}-\frac12\pas{c-1}c}
	\frac{
		\qpoch{q^{4b+4}}{q^4}{c-2b-1}
	}{
		\qpoch{q^2}{q^2}{c-2b-1}
	} \label{eq:quarter1}.
\end{align}
Here, we shall only consider \eqref{eq:quarter1}:
Lais consideration of ``increasing'' quartered hexagons would correspond to replacing
$b\to b+x$ and $c\to c+2x$ for some nonnegative integer $x$.
From \eqref{eq:quarter1} we see that
the generating function of all lattice paths with initial point $\pas{2x+1,x}$ and terminal point $\pas{2x+a,0}$
is zero for $a<1$, otherwise it is
$$
\frac{2^{1-a} q^{-2 a x-\frac{1}{2} (a-1) a+2 x} \left(q^{4 x+4};q^4\right){}_{a-1}}{\left(q^2;q^2\right){}_{a-1}}.
$$
Hence the generating function of all families of $m$ non--intersecting lattice paths with
\bit
\item initial points $\pas{2\pas{x+i-1}+1,x+i-1}$, $i=1,2,\dots m$ 
\item and terminal points  $\pas{2x+a_j,0}$, $j=1,2,\dots m$, 
\eit
where ${\mathbf a} = \pas{a_j}_{j=1}^m$ is an increasing sequence of
nonnegative integers, is given as
\begin{multline}
\label{eq:quarter-odd-gf}
2^{m^2-\sum _{j=1}^m a_j} q^{\frac{1}{2} \sum _{j=1}^m \left(a_j-a_j^2\right)+\frac{1}{6} \left(4 m^3-3 m^2-m\right)} \left(q^{2 x}\right)^{m^2-\sum_{j=1}^m a_j}\\
\times
\det\pas{
\Iverson{a_j\geq 2i-1}\cdot
\frac{
\qpoch{q^{4i+4x}}{q^4}{a_j+1-2i}
}{
\qpoch{q^{2}}{q^2}{a_j+1-2i}
}
}_{i,j=1}^{m}
\end{multline}

Here, we used Iverson's bracket 
$$
\Iverson{\text{some assertion}} =
\begin{cases}
1 &: \text{ if the assertion is true}, \\
0 &: \text{ otherwise}
\end{cases}
$$
to point out the obvious condition that a lattice path cannot have its end point to the left of its starting point;
but note that by \eqref{eq:qpochminus}, we could also \EM{omit} Iverson's bracket.

Clearly, the crucial point in \eqref{eq:quarter-odd-gf} is the 
\EM{determinant}: Observe that the determinant is \EM{zero} if and only if $a_j< 2j-1$ for some $j$, $1\leq j\leq m$. 

\secA{Evaluation of the determinant}
\begin{dfn}
We call a finite sequence of integers $\mathbf a = \pas{a_1,a_2,\dots,a_m}$ 
\EM{admissible}
if it is \EM{strictly increasing} and has the additional property 
$$2i-1\leq a_i\leq2m\text{ for all }i, 1\leq i\leq m.$$
We call an admissible seqence \EM{irreducible} if it obeys the stricter condion
$$2i+1\leq a_i\leq2m\text{ for all }i, 1\leq i< m, \text{ and } a_m=2m.$$
\end{dfn}
\begin{rem}
Admissible sequences are yet another type of objects enumerated by the ubiquituous Catalan--numbers:
The number of admissible sequences of length $m-1$ is the Catalan number $C_{m}=\frac1{m+1}\binom{2m}{m}$.
One way to prove this is by giving a simple
bijection with Dyck paths of length $2m$ (i.e., with lattice
paths from $\pas{0,0}$ to $\pas{2m,0}$ which consist of diagonal upwards and downwards steps, but never
fall below level $0$), which can be easily ``seen'' by just looking at a picture; see \figref{fig:dyck-path}.
%
\end{rem}
\begin{figure}%
\begin{center}%
\input graphics/dyck-path%
\end{center}%
{{\small The picture illustrates the bijection for the admissible sequence $\mfseq=\pas{3,6,7,8}$ of length $4$ and a Dyck
path of length $2\cdot 5=10$: For every down--step of the Dyck path except the last, draw a rectangle with one side
equal to this down--step and one corner ``leaning against the vertical axis''; the heights of these rectangles give an admissible sequence; and any admissible sequence ``encodes'' a Dyck path.}}
\caption{Illustration of the bijection between Dyck paths of length $2m$ and admissible sequences of length $m-1$.}%
\label{fig:dyck-path}%
\end{figure}%

\begin{thm}
\label{thm:det-identity}
Let $\mathbf a = \pas{a_1,a_2,\dots,a_m}$ be an \EM{admissible} sequence,
and define two $m\times m$--matrices
$\mfmat{\mfseq,x}$ and $\mfsimple\mfseq$ as follows:
\begin{equation}
\label{eq:def-two-matrices}
\mfmat{\mfseq,x}
\defeq
\pas{\frac{
\qpoch{q^{2i+2x}}{q^2}{a_j+1-2i}
}{
\qpoch{q}{q}{a_j+1-2i}
}
}_{i,j=1}^{m}
\text{ and }
\mfsimple\mfseq
\defeq
\pas{\frac{
1
}{
\qpoch{q}{q}{a_j+1-2i}
}
}_{i,j=1}^{m}
\end{equation}
(Note that the entries in \eqref{eq:def-two-matrices} are \EM{zero} whenever $a_j<2i-1$.)

Then we have
\begin{equation}
\label{eq:det-identity}
\det\of{\mfmat{\mfseq,x}} =
\pas{\prod_{k=1}^m\qpoch{q^{2x+2k}}{q}{a_k-2k+1}}
\det\of{\mfsimple\mfseq}.
\end{equation}
\end{thm}
Observe that substituting $q\to q^2$ in $\mfmat{\mfseq,x}$ gives precisely the determinant
in \eqref{eq:quarter-odd-gf}, so by the weight–preserving bijection between lozenge tilings
and families of non–intersecting lattice paths, 
\eqref{eq:det-identity} yields Lai's observation
\cite[Theorem 1.3]{Lai:2020:ROTGFOSAQH}.

\begin{rem}
The condition $a_m\leq 2m$ in \thmref{thm:det-identity} is crucial; otherwise the assertion is wrong, already for $m=1$.
\end{rem}

The rest of this section is devoted to the proof of \thmref{thm:det-identity}: It consists of several steps,
which we indicate by corresponding headlines.

\step{Reduce the general determinant evaluation to irreducible sequences}
Observe that if there is some $1\leq k<m$ such that $a_k\leq 2k$,
then the $\pas{i,j}$--entry of $\mfmat{\mfseq,x}$ 
is zero for $i>k$ and $j<k$, so the determinant can be written as the product of two minors consisting
\bit
\item of the first $k$ rows and columns of $\mfmat{\mfseq,x}$,
\item and of the last $m-k$ rows and columns of $\mfmat{\mfseq,x}$,
\eit
respectively. It is easy to see
that these minors correspond to the lists $\mfseq^\prime=\pas{a_1,\dots,a_k}$ and
$\mfseq^\pprime=\pas{a_{k+1}-2k,\dots,a_{m}-2k}$, which gives the following factorization:
\begin{equation*}
\det\of{\mfmat{\mfseq,x}} = 
\det\of{\mfmat{\mfseq^\prime,x}}
\cdot
\det\of{\mfmat{\mfseq^\pprime,x+k}}.
\end{equation*}
Both sequences $\mfseq^\prime$ and $\mfseq^\pprime$ are admissible, and if we manage to prove \thmref{thm:det-identity}
for the \EM{irreducible} case
$$a_k\geq 2k+1\text{ for all }k=1,\dots,m-1,$$
it is easy to see that  the factorization above gives the general assertion \eqref{eq:det-identity} for all admissible sequences (by induction
on the number of ``irreducible factors'').

Note that an irreducible sequence $\mfseq$ of length $m$ necessarly ends with
$$\mfseq=\pas{\dots, a_{m-1}=2m-1,a_m=2m}.$$

\step{Pull out common factors from rows and columns}
We may rewrite the factors of the product in \eqref{eq:det-identity} as follows
$$
\prod_{k=1}^m\qpoch{q^{2x+2k}}{q}{a_k-2k+1}
=\frac{\prod_{j=1}^m\qpoch{q^{2x}}{q}{a_j+1}}{\prod_{i=1}^m\qpoch{q^{2x}}{q}{2i}}
$$
and pull them out from the rows $i$ and columns $j$ of
$\mfmat{\mfseq,x}$. This operation gives a new $m\times m$--matrix $\mfmatprime{\mfseq,x}$ with $\pas{i,j}$--entry
\begin{equation}
\label{eq:simplified-entry}
\frac{
\qpoch{q^{2x}}{q}{2i}\qpoch{q^{2i+2x}}{q^2}{a_j+1-2i}
}{
\qpoch{q^{2x}}{q}{a_j+1}\qpoch{q}{q}{a_j+1-2i}
}
=
\frac{
\qpoch{q^{2x+2i}}{q^2}{a_j+1-2i}
}{
\qpoch{q^{2x+2i}}{q}{a_j+1-2i}\qpoch{q}{q}{a_j+1-2i}
}, 
\end{equation}
and \eqref{eq:det-identity} clearly is equivalent to
\begin{equation}
\label{eq:equivalent-det-identity}
\det\of{\mfmatprime{\mfseq,x}} = \det\of{\mfsimple\mfseq}.
\end{equation}

\step{The modified matrix $\mfmatprime{\mfseq,x}$ can be triangulating by a lower triangular matrix $\mffac\mfseq$ with entries from the field $\C\of q$ of rational functions in $q$}
Denote by $\ratfield$ the field $\C\of q$ of rational functions in $q$.
\begin{lem}
\label{lem:triangulization}
Let $\mfseq$ be an irreducible sequence, and consider $\mfmatprime{\mfseq,x}$ as defined in \eqref{eq:simplified-entry}.
Then there exists a \EM{lower triangular} $m\times m$--matrix $\mffac\mfseq$
whose entries are \EM{constants} from $\ratfield$ 
(i.e., the entries \EM{do not depend} on $x$) with all entries
on the main diagonal equal to $1$, such that the matrix product
\begin{equation}
\label{eq:mat-prod}
\mfmatprime{\mfseq,x}\cdot\mffac\mfseq
\end{equation}
is \EM{upper triangular}, with \EM{constants} from $\ratfield$ on its
main diagonal.
\end{lem}
If we manage to prove \lemref{lem:triangulization}, then by the multiplicativity of the determinant, we would obtain that
$$\det\of{\mfmatprime{\mfseq,x}\cdot\mffac\mfseq} = \det\of{\mfmatprime{\mfseq,x}}\cdot\det\of{\mffac\mfseq} = \det\of{\mfmatprime{\mfseq,x}},$$
and this determinant is equal to the product of the entries on the main diagonal of \eqref{eq:mat-prod}; so, in particular,
it \EM{does not depend on $x$}.
(This would already be sufficient to deduce Lai's observation concerning the quotient of generating functions \cite[Theorem 1.3]{Lai:2020:ROTGFOSAQH}.)
Moreover, for $q\in\C$ with $\abs q < 1$, \eqref{eq:equivalent-det-identity} would also hold \EM{in the limit $x\to\infty$},
so \eqref{eq:det-identity} follows since
$$
\lim_{x\to\infty}
\frac{
\qpoch{q^{2x+2i}}{q^2}{a_j+1-2i}
}{
\qpoch{q^{2x+2i}}{q}{a_j+1-2i}\qpoch{q}{q}{a_j+1-2i}
}
=
\frac{
1}{
\qpoch{q}{q}{a_j+1-2i}
}
\text{ for } \abs q < 1.
$$
So, the proof of \thmref{thm:det-identity} would be \EM{complete} if we can prove \lemref{lem:triangulization}.

\step{Finding the triangulating matrix $\mffac\mfseq$ amounts to solving $m-1$ systems of linear equations}
For some matrix $\mathbf n$, let us denote by $\mainminor{\mathbf n}st$ the \EM{submatrix} of $\mathbf n$ consisting
\bit
\item
of all rows $i\geq s$
\item
and all columns $j\geq t$.
\eit

The $k$--th column vector in a \EM{lower triangular} $m\times m$--matrix has, by definition, zero entries at
all positions $j<k$, so determining the sub--vector ${\mathbf v}_k$ corresponding to the ``interesting'' entries
of this column vector in the \EM{lower triangular} $m\times m$--matrix $\mffac\mfseq$
amounts to finding a solution of the homogeneous system of linear equations
$$
\mainminor{\mfmatprime{\mfseq,x}}{k+1}k \cdot {\mathbf v}_k = {\mathbf 0},
$$ 
which has the following \EM{additional properties}:
\begin{enumerate}
\item the entries in ${\mathbf v}_k$ are elements of $\ratfield$ which \EM{do not depend on $x$}: We
	call such solution \EM{$x$--invariant},
\item the first entry of ${\mathbf v}_k$ is not zero, whence we can divide ${\mathbf v}_k$ by this first entry (so
	the lower triangular matrix $\mffac\mfseq$ would have entries equal to $1$ on its main diagonal, as desired),
\item the inner product of the first row of $\mainminor{\mfmatprime{\mfseq,x}}kk$ with ${\mathbf v}_k$ is 
	in $\ratfield$ and \EM{does not depend on $x$}, too.
\end{enumerate}

Clearly, we have to find such solutions for $k=1,2,\dots,m-1$.


\step{Observation: It suffices to consider a \EM{single} system of linear equations}

It is easy to see that the submatrix $\mainminor{\mfmatprime{\mfseq,x}}22$
is precisely the $\pas{m-1}\times \pas{m-1}$--matrix $\mfmatprime{\overline{\mfseq},x+1}$,
where 
$$
\overline{\mfseq} = \pas{a_2-2,a_3-2,\dots,a_m-2}.
$$
Hence it suffices to consider only the \EM{first} column vector ${\mathbf v}_1$ of the matrix $\mffac{\mfseq}$.
This means that we have to show that the system of equations 
\begin{equation}
\label{eq:mprime-system}
\mainminor{\mfmatprime{\mfseq,x}}21 \cdot {\mathbf v}_1 = {\mathbf 0}
\end{equation}
has a solution with the additional properties 1, 2 and 3 from above.

\step{Express the single system of linear equations more conveniently}

Observe that we can express property 2 more conveniently: \EM{Prepend} $1$ to the \EM{irreducible} sequence
$\mfseq$,
i.e., consider the (admissible!) sequence
$$
\mfseq^\star \defeq\pas{1, a_1, \dots, a_m}
$$
(so $a^\star_1=1$ and $a^\star_k = a_{k-1}$ for $2\leq k\leq m+1$),
and define the $m\times\pas{m+1}$--matrix $\mfmatprime{\mfseq^\star,x}$ with entries as in \eqref{eq:simplified-entry},
i.e.:
\begin{equation}
\label{eq:simplified-starred-entry}
\mfmatprime{\mfseq^\star,x}_{i,j}
=
\frac{
\qpoch{q^{2x+2i}}{q^2}{a^\star_j+1-2i}
}{
\qpoch{q^{2x+2i}}{q}{a^\star_j+1-2i}\qpoch{q}{q}{a^\star_j+1-2i}
}.
\end{equation}
Note that $\mfmatprime{\mfseq^\star,x}_{i,1}=\Iverson{i=1}$: Hence an \EM{$x$--invariant} solution of
\begin{equation}
\label{eq:starred-equation}
\mfmatprime{\mfseq^\star,x}\cdot{\mathbf v}_1={\mathbf 0}
\end{equation}
is an $x$--invariant solution of \eqref{eq:mprime-system} with property 2.

Assume that we can show that \eqref{eq:starred-equation} has an $x$--invariant solution: Then such solutions
would determine a matrix $\mffac\mfseq$ which \EM{might have} zeros on the main diagonal; but this cannot
happen since we \EM{know} that $\det\of{\mfmatprime{\mfseq,x}}\neq0$ (as it is the generating
function of a \EM{non--empty} set of lattice paths). This means that such solutions ``automatically'' have property 3:
So the proof of \lemref{lem:triangulization} boils down to showing that \eqref{eq:starred-equation} has an $x$--invariant solution.
%

\step{Divide all equations by the coefficient corresponding to the last column}
The system \eqref{eq:starred-equation} features $m$ equations for $m+1$ variables, so for every
fixed $x$ we expect a solution space of dimension $\geq 1$. 
As pointed out above, element $a_m$ equals $2m$ in an \EM{irreducible} sequence of length $m$.
By dividing the $i$--the equation (i.e., row $i$ of the 
matrix $\mfmatprime{\mfseq^\star,x}$) by the (non--zero) coefficient
corresponding to $a^\star_{m+1}=a_m=2m$
$$
\frac{
\qpoch{q^{2x+2i}}{q^2}{2m+1-2i}
}{
\qpoch{q^{2x+2i}}{q}{2m+1-2i}\qpoch{q}{q}{2m+1-2i}
}
$$
we do not change the space of solutions, but obtain a modified coefficient matrix $\mfmatpprime{\mfseq^\star,x}$
with $\pas{i,j}$--entry
\begin{multline}
\label{eq:modified-matrix}
\frac{
\qpoch{q^{2x+2i}}{q^2}{a_j+1-2i}
\qpoch{q^{2x+2i}}{q}{2m+1-2i}\qpoch{q}{q}{2m+1-2i}
}{
\qpoch{q^{2x+2i}}{q^2}{2m+1-2i}
\qpoch{q^{2x+2i}}{q}{a_j+1-2i}\qpoch{q}{q}{a_j+1-2i}
}
\\ =
\frac{
\qpoch{q^{2+a_j-2i}}{q}{2m-a_j}\qpoch{q^{1+a_j+2x}}{q}{2m-a_j}
}{
\qpoch{q^{2+2a_j+2x-2i}}{q^2}{2m-a_j}
}.
\end{multline}

\step{Consider a single system of linear equations with \EM{more} variables}
It is convenient to view the coefficient matrix $\mfmatpprime{\mfseq^\star,x}$ 
as the \EM{submatrix} corresponding to columns $a^\star_1,a^\star_2, \dots, a^\star_{m+1}$
of the $m\times2m$--matrix with $\pas{i,j}$--entry
$$
\frac{
\qpoch{q^{2+j-2i}}{q}{2m-j}\qpoch{q^{1+j+2x}}{q}{2m-j}
}{
\qpoch{q^{2+2j+2x-2i}}{q^2}{2m-j}
},
$$
and to reverse the order of columns, i.e., to consider the $m\times 2m$--matrix $\mfmatppprime\of{m,x}$
with $\pas{i,j}$--entry
\begin{equation}
\label{eq:mppp}
\mfentry{m,i,j,x}\defeq
\frac{
\qpoch{q^{2\pas{m-i+1}-j}}{q}{j}\qpoch{q^{2\pas{m+x}-j+1}}{q}{j}
}{
\qpoch{q^{2\pas{2m+x-i+1}-2j}}{q^2}{j}
},
\end{equation}
where $1\leq i\leq m$ and $0\leq j \leq 2m-1$.

Observe that the factor $\qpoch{q^{2\pas{m-i+1}-j}}{q}{j}$ in \eqref{eq:mppp} is \EM{zero} for $j\geq2\pas{m-i+1}$, hence
the matrix $\mfmatppprime\of{m,x}$ has a ``staircase shape'', with all entries of row $i$ 
equal to \EM{zero} for $j\geq2\pas{m-i+1}$. To illustrate this, we present $\mfmatppprime\of{3,x}$ (after expansion of the
$q$--Pochhammer symbols involving the parameter $x$ \EM{and cancellation}; for typesetting
reasons, we give the \EM{transpose} of this matrix):
{\small
\begin{equation}
\label{eq:ex-mppp}
\pas{\mfmatppprime\of{3,x}}^T =
\begin{pmatrix}
 1 &  1 &  1 \\
 \frac{\qpoch{q^5}q1 \left(1-q^{2 x+6}\right)}{1-q^{2 x+10}} 
 &  \frac{\qpoch{q^3}q1 \left(1-q^{2 x+6}\right)}{1-q^{2 x+8}}
 &  1-q \\
 \frac{\qpoch{q^4}q2 \left(1-q^{2 x+5}\right) \left(1-q^{2 x+6}\right)}{\left(1-q^{2 x+8}\right) \left(1-q^{2 x+10}\right)} & 
  \frac{\qpoch{q^2}q2 \left(1-q^{2 x+5}\right)}{1-q^{2 x+8}} &  0 \\
 \frac{\qpoch{q^3}q3 \left(1-q^{2 x+4}\right) \left(1-q^{2 x+5}\right)}{\left(1-q^{2 x+8}\right) \left(1-q^{2 x+10}\right)} &  \frac{\qpoch{q}q3 \left(1-q^{2 x+5}\right)}{1-q^{2 x+8}} &  0 \\
 \frac{\qpoch{q^2}q4 \left(1-q^{2 x+3}\right) \left(1-q^{2 x+5}\right)}{\left(1-q^{2 x+8}\right) \left(1-q^{2 x+10}\right)} &  0
   &  0 \\
 \frac{\qpoch{q}q5 \left(1-q^{2 x+3}\right) \left(1-q^{2 x+5}\right)}{\left(1-q^{2 x+8}\right) \left(1-q^{2
   x+10}\right)} &  0 &  0
\end{pmatrix}
\end{equation}
}

\step{A useful observation concerning \eqref{eq:mppp}}


The attentive reader may have noticed a phenomenon in the example \eqref{eq:ex-mppp}: The number
of factors in the denominator is \EM{not increasing all the way}, but apparently becomes \EM{constant} from
$j=m-i$ on. This is no coincidence, but due to certain \EM{cancellations}: Observe that the factors depending on $x$ in the numerator of \eqref{eq:mppp}
end with $\pas{1-q^{2m+2x}}$, 
and from $j=m-i+1+t$ on ($t=0,1,\dots$), the
``additional factors'' $\pas{1-q^{2m+2x-2t}}$ appearing in the denominator cancel out with corresponding factors in
the numerator, such that in the last two entries of each row only factors with \EM{odd} ``$x$--depending $q$--exponents''
$\pas{1-q^{2x+2t-1}}$ survive the cancellation in the numerator; moreover, the last entry is $\pas{1-q}$--times the next--to--last entry.
Altogether, this implies: If we multiply every row $i$ of $\mfmatppprime\of{m,x}$ 
by the denominator of its last non--zero
entry
$$\qpoch{q^{2x+2m+2}}{q^2}{m-i},$$
then the $\pas{i,j}$--entry of the resulting matrix is equal to $\qpoch{q^{2\pas{m-i+1}-j}}{q}{j}$ times a product of $\pas{m-i}$ factors
of the form $\pas{1-q^{2x+y}}$, i.e.,
\begin{equation}
\label{eq:funfact1}
\mfmatppprime\of{m,x}_{i,j}\cdot\qpoch{q^{2x+2m+2}}{q^2}{m-i}
=
\qpoch{q^{2\pas{m-i+1}-j}}{q}{j}\cdot\prod_{k=1}^{m-i}\pas{1-q^{2z+y_{j,k}}} 
\end{equation}
for certain integers $y_{j,k}$.

\step{Claim: The first homogeneous equation of \eqref{eq:starred-equation} actually \EM{has} a non--trivial $x$--invariant solution}

It might not be clear at first sight that there is a non--trivial $x$--invariant solution \EM{at all}; even for a \EM{single} equation
of \eqref{eq:starred-equation}:
In order to show that this is indeed the case, we state and prove a little Lemma.

Recall the definition of the $q$--binomial coefficient
\begin{equation}
\qbinom nk \defeq\frac{\qpoch qqn}{\qpoch qqk \qpoch qq{n-k}}
\end{equation}
for non--negative integers $n,k$ with $0\leq k\leq n$, and the well--known recursions
\begin{align}
\qbinom nk &= \qbinom{n-1}{k-1} + q^k\qbinom{n-1}{k}, \label{eq:qbinom-rec-1} \\
\qbinom nk &= q^{n-k}\qbinom{n-1}{k-1} + \qbinom{n-1}{k}. \label{eq:qbinom-rec-2}
\end{align}
An easy consequence of these recursions is the well--known identity
\begin{equation}
\label{eq:1-1^n}
\sum_{k=0}^n \pas{-1}^k q^{\frac{k\pas{k+1}}2 - k n}\qbinom nk = \Iverson{n=0}.
\end{equation}
Now define
$$
e\of{n,k}\defeq \frac{k\pas{k+1}}2 - k n
$$
and observe that \eqref{eq:1-1^n} can be generalized as follows:
\begin{lem}
\label{lem:qbinom-identity}
Let $n,r$ be positive integers, and let $\pas{\gamma_1,\gamma_2,\dots,\gamma_{n-r}}$ be an (arbitrary) sequence of $n-r$ complex numbers.
Then we have: 
\begin{equation}
\label{eq:qbinom-identity}
\sum_{k=0}^n\pas{-1}^k q^{\qexp nk}\qbinom nk\prod_{i=1}^{n-r}\pas{1-\gamma_i\cdot q^{k}} = 0.
\end{equation}
\end{lem}
\begin{proof}
Let $r$ be arbitrary but fixed, and proceed by induction on $n$: Clearly, for $n\leq r$ the statement is true, since
it is equivalent to \eqref{eq:1-1^n} in this case.


For the inductive step $\pas{n-1}\to n$, we
set $f\defeq\prod_{i=1}^{n-1-r}\pas{1-\gamma_i\cdot q^{k}}$ and
use the recursions \eqref{eq:qbinom-rec-1}
and \eqref{eq:qbinom-rec-2} to rewrite the unsigned $k$--th summand in \eqref{eq:qbinom-identity} as
{\small
\begin{equation*}
q^{e\of{n,k}}\pas{
\pas{
\qbinom{n-1}k q^k + \qbinom{n-1}{k-1}
}
-
\gamma_{n-r}\cdot q^{k}
\pas{
\qbinom{n-1}k  + \qbinom{n-1}{k-1}q^{n-k}
}
}
f.
\end{equation*}
}

Using the obvious relations
$$
e\of{n,k} + k = e\of{n-1,k} \text{ and }
e\of{n,k} + n - 1 = e\of{n-1,k-1},
$$
we may rewrite this
as
\begin{multline*}
\pas{
\qbinom{n-1}k q^{e\of{n-1,k}} + q^{1-n}\qbinom{n-1}{k-1}q^{e\of{n-1,k-1}}
}
f
\\
-
{\gamma_{n-r}}
\pas{
\qbinom{n-1}k q^{e\of{n-1,k}}  + q\qbinom{n-1}{k-1}q^{e\of{n-1,k-1}}
}
f,
\end{multline*}
and the assertion follows by induction.
\end{proof}
From this Lemma, we can deduce immediately the following Corollary.
\begin{cor}
\label{cor:x_independent_solution}
Let $n,r$ be positive integers, and assume that for each for $j=1,2,\dots,n+1$ we are
given
\bit
\item $n-r$ complex numbers $\gamma_{j,1},\gamma_{j,2},\dots,\gamma_{j,n-r}$
\item and a rational function $c_j = c_j\of q$ in $\ratfield$.
\eit

Then each $n+1$ subsequent rows $i=i_0+1,i_0+2,\dots,i_0+n+1$ (for arbitrary $i_0\in\N$) of the form
\begin{equation}
\label{eq:row}
\pas{c_j\prod_{k=1}^{n-r}\pas{1-\gamma_{j,k}\, q^{i}}}_{j=1}^{n+1}
\end{equation}
are \EM{not linearly independent} over $\ratfield$ 
(by \lemref{lem:qbinom-identity}),
and therefore \EM{the same} is true for the \EM{columns} of the matrix
\begin{equation}
\label{eq:cor-mat}
\pas{c_j\prod_{k=1}^{n-r}\pas{1-\gamma_{j,k}\, q^{i}}}_{\pas{i,j}=\pas{i_0+1,1}}^{\pas{i_0+n+1,n+1}}
\end{equation}

Hence the homogeneous system of linear equations corresponding to this matrix
has a \EM{nontrivial} solution $\mathbf b=\pas{b_1,b_2,\dots,b_{n+1}}$ in $\ratfield$. 

Since such solution $\mathbf b$ fulfils the $n$ subsequent
equations corresponding to rows $i_0+2,i_0+3,\dots,i_0+n+1$, it also fulfils the equation
corresponding to row $i_0+n+2$ (since this row is a linear combination of its
$n$ predecessors, by \lemref{lem:qbinom-identity}) and the equation corresponding to row $i_0+0$
(since this is a linear combination of its $n$ successors, again by \lemref{lem:qbinom-identity}).

In short: \EM{Every} solution $\mathbf b$ of the homogeneous system of linear equations corresponding to \eqref{eq:cor-mat}
fulfils the equations corresponding to
rows of the form \eqref{eq:row} \EM{for all} $i\in\N$, and in that sense is
\EM{$i$--invariant}. Therefore we have for all complex numbers $q$ with $\abs q< 1$
$$
0 = \lim_{i\to\infty} \sum_{j=1}^{n+1}c_j\pas{\prod_{k=1}^{n-r}\pas{1-\gamma_{j,k}\, q^{i}}}  b_j
= \sum_{j=1}^{n+1}{c_j  b_j},
$$
so $\mathbf b$ is a solution to the simpler equation
$$
\sum_{j=1}^{n+1}{c_j  b_j} = 0.
$$
\end{cor}
Now observe that 
\corref{cor:x_independent_solution} \EM{applies} to the first row of $\mfmatpprime{\mfseq^\star,x}$
after multiplication by the denominator of its last non--zero entry (as in \eqref{eq:funfact1}): 
The $j$--th entry in this row is equal to $\qpoch{q^{2m-j}}{q}{j}$ times a product of $m-1$ factors of the form $\pas{1-q^{2x+y}}$, and the length of this row is $m+1$;
hence we may consider the $m+1$ equations for $x,x+1/2,\dots,x+m/2$ and deduce that there exists an $x$--invariant solution for the \EM{first} equation of \eqref{eq:starred-equation}.

\step{Another useful observation concerning \eqref{eq:mppp}}



We set $x=\frac12-k$ in $\mfentry{m,i,j,x}$ and observe:
\begin{align}
\mfentry{m,i,j,\frac12-k}
& =
\frac{
\qpoch{q^{2\pas{m-i+1}-j}}{q}{j}\qpoch{q^{2\pas{m-k+1}-j}}{q}{j}
}{
\qpoch{q^{2\pas{2m-k+\frac12-i+1}-2j}}{q^2}{j}
} \notag \\
& =
\frac{
\qpoch{q^{2\pas{m-k+1}-j}}{q}{j}\qpoch{q^{2\pas{m+\pas{\frac12-i}}-j+1}}{q}{j}
}{
\qpoch{q^{2\pas{2m+\pas{\frac12-i}-k+1}-2j}}{q^2}{j}
} \notag \\
&= \mfentry{m,k,j,1/2-i}. \label{eq:halfinteger-evaluation}
\end{align}

Note that this means: Any \EM{$x$--invariant} solution of the equation corresponding to row $k$ of $\mfmatpprime{\mfseq^\star,x}$
is also
a solution of \EM{all rows} $i\geq k$, but for \EM{fixed} $x=\frac12-k$.

%
%

\step{Put everything together}
For some subsequence $C\subseteq\setof{0,1,\dots,2m-1}$ of columns of $\mfmatppprime\of{m,x}$, we may consider
the homogeneous system of linear equations corresponding to the submatrix of $\mfmatppprime\of{m,x}$ which consists
of the columns in $C$: We call this system the \EM{equations corresponding to $C$}, and if we consider
only the last $k$ ($1\leq k\leq m$) of these equations (corresponding to the submatrix of $\mfmatppprime\of{m,x}$ which consists
of the rows $m-k+1,m-k+2,\dots,m$ and the columns in $C$), we call this homogeneous system of equations the
\EM{last $k$ equations corresponding to $C$}.
For arbitrary but fixed $x$, the set of solutions of the last $k$ equations corresponding to some $C$
is a \EM{linear subspace} of $\ratfield^{\abs C}$, which \EM{contains} the set of \EM{$x$--invariant solutions}, and the
latter is a \EM{linear subspace} of $\ratfield^{\abs C}$, too. 


Recall that we have to show that the equations corresponding to 
$\mfseq^\star$ have a
\EM{one--dimensional $x$--invariant space of solutions}:
We shall prove this by working our way up from the last row $m$ to the first.

More precisely, we claim
that the last $k$ rows of the equations corresponding to $\mfseq^\star$ 
\bit
\item have an $x$--invariant solutions space whose
dimension is equal
\bit
\item to the number of elements $j$ in $\mfseq^\star$ such that $j\leq 2k$
\item minus $k$,
\eit
\item and that there are no other solutions, i.e.: \EM{Every} solution (for \EM{arbitrary} fixed $x$) is, in fact $x$--invariant.
\eit
This assertion is immediately
clear for $k=1$, since the (single) last equation corresponding to $\mfseq^\star$ has a $1$--dimensional solution space, which
is $x$--invariant since this equation does not contain $x$ at all (see example \eqref{eq:ex-mppp}):
\begin{equation}
\label{eq:last-equation}
1\cdot c_0 + \pas{1-q}\cdot c_1 = 0 \iff\pas{c_0,c_1}\in
\setof{\lambda\cdot\pas{q-1,1}:\lambda\in\C}.
\end{equation}
So assume we proved our claim for the last $k$ rows
$$
m-k+1, m-k+2, \dots, m
$$
of the equations corresponding to $\mfseq^\star$, and assume that for the
next row $m-k\geq 1$ there are $l$ columns of $\mfseq^\star$ which are $\leq 2\pas{k+1}$ (note that $l\geq k+1$, since
$\mfseq^\star$ is an \EM{admissible} sequence). Now recall \eqref{eq:funfact1}: The entries of row $m-k$ are, after multiplication with
the denominator of the entry $\mfmatppprime\of{m,x}_{m-k,2\pas{m-k}-1}$, equal to some element in $\ratfield$ times a product of
$k$ factors of the form $\pas{1-q^{2x+y}}$. So by \corref{cor:x_independent_solution}, row $m-k$ in the equations
corresponding to
\bit
\item the \EM{first $k$ columns} of $\mfseq^\star$,
\item plus one of the columns $a_{k+1}, a_{k+2}, \dots, a_l$ which are less or equal to $2k+1$
\eit
has a non--trivial $x$--invariant solution, and \EM{every} linear combination of these $\pas{l-k}$ $x$--invariant solutions is
again an $x$--invariant solution, so the dimension of the linear space of $x$--invariant solutions is at least
$l-k$. Now we combine two observations:
\begin{enumerate} 

\item Since the first $k$ entries of $\mfseq^\star$ are an admissible sequence, the submatrix of $\mfmatppprime\of{m,x}$
consisting of
\bit
\item the last $k$ rows
\item and the first $k$ columns of $\mfseq^\prime$
\eit
has determinant $\neq 0$. Hence the $k$ equations are linearly independent,
and the dimension of the solution space for arbitrary, but fixed $x$ is $l-k$: But this implies that \EM{all} solutions
(for \EM{any} $x$) are, in fact, $x$--invariant.

\item By observation \eqref{eq:halfinteger-evaluation}, for $i\leq k$ the following rows are identical:
\bit
\item row $m-i$ for $x=\frac12-\pas{m-k}$,
\item row $m-k$ for $x=\frac12-\pas{m-i}$.
\eit
This implies that each $x$--invariant solution
of the $\pas{m-k}$--th equation corresponding to $\mfseq^\star$ is a solution of the last $k+1$
equations corresponding to $\mfseq^\star$ for \EM{fixed} $x=\frac12-\pas{m-k}$, and since we already know that \EM{all} solutions
(for \EM{any} $x$) of the last $k$ equations
are $x$--invariant, the same holds true for the last $k+1$ equations.
\end{enumerate}
Altogether, this finishes the proof of \lemref{lem:triangulization}, and thus of \thmref{thm:det-identity}.\hfill\qedsymbol

\secB{A simple algorithm for actually finding the solutions which constitute $\mffac\mfseq$ of \lemref{lem:triangulization}}
For the homogeneous system of the $m$ linearly independent equations in $2m$ variables corresponding to the coefficient matrix $\mfmatppprime$
\begin{equation}
\label{eq:sys-ppp}
\mfmatppprime\cdot\mathbf c =\mathbf0,
\end{equation}
we expect a solution space of dimension $m$:
We claim that this solution space is \EM{$x$--invariant} and spanned by a vector
$$\mathbf c_1=\pas{1,\alpha_1,0,\alpha_3,\dots,0,\alpha_{2m-1}}$$
which equals the $2m$ first elements of a certain infinite sequence in $\ratfield$, which can
can be \EM{constructed recursively} and which is independent of $m$,
together with $m-1$ \EM{shifts} of this vector
$$
\mathbf c_2=\pas{0,0,1,\alpha_1,0,\alpha_3,\dots,0,\alpha_{2m-3}},\dots \mathbf c_m=\pas{0,0,\dots,1,\alpha_1}.
$$
For instance, the solution space for $m=4$ is the $\ratfield$--span of the columns of the following matrix:
{\small
$$
\begin{pmatrix}
1 & 0 & 0 & 0\\
-\frac1{1-q} & 0 & 0 & 0\\
0 & 1 & 0 & 0\\
\frac{q}{\pas{1-q}^2\pas{1-q^3}} & -\frac1{1-q} & 0 & 0\\
0 & 0 & 1 & 0\\
\frac{q^2\pas{1+q^2}}{\pas{1-q}^3\pas{1-q^3}\pas{1-q^5}} & \frac{q}{\pas{1-q}^2\pas{1-q^3}} & -\frac1{1-q} \\
0 & 0 & 0 & 1\\
\frac{q^3\pas{q^8+q^7+3 q^6+2 q^5+3 q^4+2 q^3+3 q^2+q+1}}{\pas{1-q}^3\pas{1-q^3}^2\pas{1-q^5}\pas{1-q^7}}
&
\frac{q^2\pas{1+q^2}}{\pas{1-q}^3\pas{1-q^3}\pas{1-q^5}}
&
\frac{q}{\pas{1-q}^2\pas{1-q^3}}
& -\frac1{1-q}
\end{pmatrix}
$$ 
}

Since we already \EM{know} that all solutions are $x$--invariant, we may work with the simpler matrix
$\mfmatppprime\of{m,\infty}$ with $\pas{i,j}$--entry
$$
\mfentry{m,i,j,\infty} = \lim_{x\to\infty}\mfentry{m,i,j,x} =
\qpoch{q^{2\pas{m-i+1}-j}}{q}{j}
$$
(where we assumed that $\abs q < 1$ in taking the limit)
and consider 
$$\mfseq^\star=\pas{0,1,3,5,\dots,2m-1}.$$
Observe, that in each step of ``working our
way up'' from the last equation corresponding to $\mfseq^\star$ (as in the proof of \lemref{lem:triangulization}),
precisely one new variable has to be considered. I.e., starting with $\alpha_1=q-1$ from the last equation (see
\eqref{eq:last-equation}), the next--to--last equation becomes a linear equation in a single variable, which,
of course can easily be solved: This explains the recursive construction of the solution vector $\mathbf c_1$.

Now observe that deleting the first two columns and the last row of $\mfmatppprime\of{m,\infty}$,
and dividing all remaining rows by their first entry gives $\mfmatppprime\of{m-1,\infty}$,
whence we may prepend two zeros to the solution $\mathbf c_1$ (as just described) for $m-1$ and obtain another
solution of $\mfmatppprime\of{m,\infty}$: This makes clear that we find the $m$ solution
vectors $\mathbf c_1,\dots,\mathbf c_m$ (which obviously are linearly independent), as claimed above.

We already know that there exists an $x$--invariant solution of the equations corresponding to general
$\mfseq^\star$ (derived from some admissible sequence $\mfseq$, as in the proof of \lemref{lem:triangulization}), which may be viewed as a solution vector for \eqref{eq:sys-ppp} where all entries
not in $\mfseq^\star$ are set to zero: It is easy to see how to construct a solution vector with these ``prescribed zeros''
as a linear combination of $\mathbf c_1,\dots,\mathbf c_m$.

\section*{Acknowledgement}
I am grateful to Christian Krattenthaler for helpful discussions.

\bibliography{/Users/mfulmek/Informations/TeX/database}

\end{document}

%% file: config.tex
\newif\ifenglish

\englishtrue

\newif\ifvariant
\variantfalse

\def\bit{\begin{itemize}}
\def\eit{\end{itemize}}
\def\beq{\begin{equation}}
\def\eeq{\end{equation}}

%% file: mf_makros.tex
\def\of#1{\left(#1\right)} 
\def\pas#1{\left(#1\right)} 
\def\brk#1{\left[#1\right]} 
\def\setof#1{\left\{#1\right\}}

\def\abs#1{{\left\lvert{#1}\right\rvert}}

\def\defeq{\stackrel{\text{\tiny def}}{=}}

\def\C{{\mathbb C}}

\def\F{{\mathbb F}}

\def\N{{\mathbb N}}

\def\Z{{\mathbb Z}}

\def\0{{\mathbf 0}} 
\def\1{{\mathbf 1}} 

\def\Iverson#1{\left[#1\right]}

\def\CT1{CT1}
\def\xxxCT1{CT1}

%% file: mf-theorems.tex
\newtheorem{thm}{\ifenglish Theorem\else Satz\fi}[section] 

\newtheorem{lem}[thm]{Lemma}
\newtheorem{rem}[thm]{\ifenglish Remark\else Bemerkung\fi}
\newtheorem{dfn}[thm]{Definition}
\newtheorem{cor}[thm]{\ifenglish Corollary\else Korollar\fi}

\newtheoremstyle{excstyle}
  {1em}
  {2pt}
  {\sffamily\footnotesize\slshape}
  {0pt}
  {\sffamily\footnotesize\bfseries}
  {:}
  { }
  {}
\theoremstyle{excstyle}

\def\thmref#1{\ifenglish Theorem\else Satz\fi~\ref{#1}}
\def\lemref#1{Lemma~\ref{#1}}

\def\corref#1{\ifenglish Corollary\else Korollar\fi~\ref{#1}}

\def\figref#1{\ifenglish Figure\else Abbildung\fi~\ref{#1}}

\ifenglish\else\fi

%% file: graphics/lai-hex.tex
\psset{unit=0.5cm}
\begin{pspicture}(-6.0,-6.5622)(6.0,4.3301)
\psset{linewidth=0.02,fillstyle=none,linecolor=black}
\psline(-2.5,0.0)(0.0,4.3301)
\psline(2.5,0.0)(0.0,4.3301)
\psline(-2.0,-0.866)(0.5,3.4641)
\psline(2.0,-0.866)(-0.5,3.4641)
\psline(-1.5,-1.7321)(1.0,2.5981)
\psline(1.5,-1.7321)(-1.0,2.5981)
\psline(-1.0,-2.5981)(1.5,1.7321)
\psline(1.0,-2.5981)(-1.5,1.7321)
\psline(-0.5,-3.4641)(2.0,0.866)
\psline(0.5,-3.4641)(-2.0,0.866)
\psline(0.0,-4.3301)(2.5,0.0)
\psline(0.0,-4.3301)(-2.5,0.0)
\psset{linecolor=lightgray}
\psline(-0.0,4.3301)(0.0,4.3301)
\psline(-0.5,3.4641)(0.5,3.4641)
\rput(0.0,3.4641){{\tiny\black $0$}}
\psline(-1.0,2.5981)(1.0,2.5981)
\rput(-0.5,2.5981){{\tiny\black $-1$}}
\rput(0.5,2.5981){{\tiny\black $1$}}
\psline(-1.5,1.7321)(1.5,1.7321)
\rput(-1.0,1.7321){{\tiny\black $-2$}}
\rput(0.0,1.7321){{\tiny\black $0$}}
\rput(1.0,1.7321){{\tiny\black $2$}}
\psline(-2.0,0.866)(2.0,0.866)
\rput(-1.5,0.866){{\tiny\black $-3$}}
\rput(-0.5,0.866){{\tiny\black $-1$}}
\rput(0.5,0.866){{\tiny\black $1$}}
\rput(1.5,0.866){{\tiny\black $3$}}
\psset{fillstyle=solid,linecolor=black}
\pspolygon[fillcolor=Apricot](-3.0,-0.866)(-2.0,-0.866)(-1.5,0.0)(-2.5,0.0)
\psline[linewidth=0.1,linecolor=white](-2.75,-0.433)(-1.75,-0.433)
\pspolygon[fillcolor=Apricot](-2.0,-0.866)(-1.0,-0.866)(-0.5,0.0)(-1.5,0.0)
\psline[linewidth=0.1,linecolor=white](-1.75,-0.433)(-0.75,-0.433)
\pspolygon[fillcolor=Tan](-1.0,-0.866)(-0.5,-1.7321)(0.0,-0.866)(-0.5,0.0)
\psline[linewidth=0.1,linecolor=white](-0.75,-0.433)(-0.25,-1.299)
\rput(-0.5,-0.866){{\tiny\black $-1$}}
\pspolygon[fillcolor=Mahogany](0.0,-0.866)(1.0,-0.866)(0.5,0.0)(-0.5,0.0)
\pspolygon[fillcolor=Mahogany](1.0,-0.866)(2.0,-0.866)(1.5,0.0)(0.5,0.0)
\pspolygon[fillcolor=Mahogany](2.0,-0.866)(3.0,-0.866)(2.5,0.0)(1.5,0.0)
\pspolygon[fillcolor=black](-3.5,-1.7321)(-2.5,-1.7321)(-3.0,-0.866)(-3.0,-0.866)
\pspolygon[fillcolor=Mahogany](-2.5,-1.7321)(-1.5,-1.7321)(-2.0,-0.866)(-3.0,-0.866)
\pspolygon[fillcolor=Mahogany](-1.5,-1.7321)(-0.5,-1.7321)(-1.0,-0.866)(-2.0,-0.866)
\pspolygon[fillcolor=Apricot](-0.5,-1.7321)(0.5,-1.7321)(1.0,-0.866)(0.0,-0.866)
\psline[linewidth=0.1,linecolor=white](-0.25,-1.299)(0.75,-1.299)
\pspolygon[fillcolor=Apricot](0.5,-1.7321)(1.5,-1.7321)(2.0,-0.866)(1.0,-0.866)
\psline[linewidth=0.1,linecolor=white](0.75,-1.299)(1.75,-1.299)
\pspolygon[fillcolor=Apricot](1.5,-1.7321)(2.5,-1.7321)(3.0,-0.866)(2.0,-0.866)
\psline[linewidth=0.1,linecolor=white](1.75,-1.299)(2.75,-1.299)
\pspolygon[fillcolor=Tan](2.5,-1.7321)(3.0,-2.5981)(3.5,-1.7321)(3.0,-0.866)
\psline[linewidth=0.1,linecolor=white](2.75,-1.299)(3.25,-2.1651)
\rput(3.0,-1.7321){{\tiny\black $6$}}
\pspolygon[fillcolor=Tan](-4.0,-2.5981)(-3.5,-3.4641)(-3.0,-2.5981)(-3.5,-1.7321)
\psline[linewidth=0.1,linecolor=white](-3.75,-2.1651)(-3.25,-3.0311)
\rput(-3.5,-2.5981){{\tiny\black $-7$}}
\pspolygon[fillcolor=Mahogany](-3.0,-2.5981)(-2.0,-2.5981)(-2.5,-1.7321)(-3.5,-1.7321)
\pspolygon[fillcolor=Mahogany](-2.0,-2.5981)(-1.0,-2.5981)(-1.5,-1.7321)(-2.5,-1.7321)
\pspolygon[fillcolor=Mahogany](-1.0,-2.5981)(0.0,-2.5981)(-0.5,-1.7321)(-1.5,-1.7321)
\pspolygon[fillcolor=Mahogany](0.0,-2.5981)(1.0,-2.5981)(0.5,-1.7321)(-0.5,-1.7321)
\pspolygon[fillcolor=Mahogany](1.0,-2.5981)(2.0,-2.5981)(1.5,-1.7321)(0.5,-1.7321)
\pspolygon[fillcolor=Mahogany](2.0,-2.5981)(3.0,-2.5981)(2.5,-1.7321)(1.5,-1.7321)
\pspolygon[fillcolor=black](3.0,-2.5981)(4.0,-2.5981)(3.5,-1.7321)(3.5,-1.7321)
\pspolygon[fillcolor=black](-4.5,-3.4641)(-3.5,-3.4641)(-4.0,-2.5981)(-4.0,-2.5981)
\pspolygon[fillcolor=Apricot](-3.5,-3.4641)(-2.5,-3.4641)(-2.0,-2.5981)(-3.0,-2.5981)
\psline[linewidth=0.1,linecolor=white](-3.25,-3.0311)(-2.25,-3.0311)
\pspolygon[fillcolor=Apricot](-2.5,-3.4641)(-1.5,-3.4641)(-1.0,-2.5981)(-2.0,-2.5981)
\psline[linewidth=0.1,linecolor=white](-2.25,-3.0311)(-1.25,-3.0311)
\pspolygon[fillcolor=Apricot](-1.5,-3.4641)(-0.5,-3.4641)(0.0,-2.5981)(-1.0,-2.5981)
\psline[linewidth=0.1,linecolor=white](-1.25,-3.0311)(-0.25,-3.0311)
\pspolygon[fillcolor=Tan](-0.5,-3.4641)(0.0,-4.3301)(0.5,-3.4641)(0.0,-2.5981)
\psline[linewidth=0.1,linecolor=white](-0.25,-3.0311)(0.25,-3.8971)
\rput(0.0,-3.4641){{\tiny\black $0$}}
\pspolygon[fillcolor=Mahogany](0.5,-3.4641)(1.5,-3.4641)(1.0,-2.5981)(0.0,-2.5981)
\pspolygon[fillcolor=Mahogany](1.5,-3.4641)(2.5,-3.4641)(2.0,-2.5981)(1.0,-2.5981)
\pspolygon[fillcolor=Mahogany](2.5,-3.4641)(3.5,-3.4641)(3.0,-2.5981)(2.0,-2.5981)
\pspolygon[fillcolor=Mahogany](3.5,-3.4641)(4.5,-3.4641)(4.0,-2.5981)(3.0,-2.5981)
\pspolygon[fillcolor=Apricot](-5.0,-4.3301)(-4.0,-4.3301)(-3.5,-3.4641)(-4.5,-3.4641)
\psline[linewidth=0.1,linecolor=white](-4.75,-3.8971)(-3.75,-3.8971)
\pspolygon[fillcolor=Tan](-4.0,-4.3301)(-3.5,-5.1962)(-3.0,-4.3301)(-3.5,-3.4641)
\psline[linewidth=0.1,linecolor=white](-3.75,-3.8971)(-3.25,-4.7631)
\rput(-3.5,-4.3301){{\tiny\black $-7$}}
\pspolygon[fillcolor=Mahogany](-3.0,-4.3301)(-2.0,-4.3301)(-2.5,-3.4641)(-3.5,-3.4641)
\pspolygon[fillcolor=Mahogany](-2.0,-4.3301)(-1.0,-4.3301)(-1.5,-3.4641)(-2.5,-3.4641)
\pspolygon[fillcolor=Mahogany](-1.0,-4.3301)(0.0,-4.3301)(-0.5,-3.4641)(-1.5,-3.4641)
\pspolygon[fillcolor=Apricot](0.0,-4.3301)(1.0,-4.3301)(1.5,-3.4641)(0.5,-3.4641)
\psline[linewidth=0.1,linecolor=white](0.25,-3.8971)(1.25,-3.8971)
\pspolygon[fillcolor=Tan](1.0,-4.3301)(1.5,-5.1962)(2.0,-4.3301)(1.5,-3.4641)
\psline[linewidth=0.1,linecolor=white](1.25,-3.8971)(1.75,-4.7631)
\rput(1.5,-4.3301){{\tiny\black $3$}}
\pspolygon[fillcolor=Mahogany](2.0,-4.3301)(3.0,-4.3301)(2.5,-3.4641)(1.5,-3.4641)
\pspolygon[fillcolor=Mahogany](3.0,-4.3301)(4.0,-4.3301)(3.5,-3.4641)(2.5,-3.4641)
\pspolygon[fillcolor=Mahogany](4.0,-4.3301)(5.0,-4.3301)(4.5,-3.4641)(3.5,-3.4641)
\pspolygon[fillcolor=black](-5.5,-5.1962)(-4.5,-5.1962)(-5.0,-4.3301)(-5.0,-4.3301)
\pspolygon[fillcolor=Mahogany](-4.5,-5.1962)(-3.5,-5.1962)(-4.0,-4.3301)(-5.0,-4.3301)
\pspolygon[fillcolor=Tan](-3.5,-5.1962)(-3.0,-6.0622)(-2.5,-5.1962)(-3.0,-4.3301)
\psline[linewidth=0.1,linecolor=white](-3.25,-4.7631)(-2.75,-5.6292)
\rput(-3.0,-5.1962){{\tiny\black $-6$}}
\pspolygon[fillcolor=Mahogany](-2.5,-5.1962)(-1.5,-5.1962)(-2.0,-4.3301)(-3.0,-4.3301)
\pspolygon[fillcolor=Mahogany](-1.5,-5.1962)(-0.5,-5.1962)(-1.0,-4.3301)(-2.0,-4.3301)
\pspolygon[fillcolor=Mahogany](-0.5,-5.1962)(0.5,-5.1962)(0.0,-4.3301)(-1.0,-4.3301)
\pspolygon[fillcolor=Mahogany](0.5,-5.1962)(1.5,-5.1962)(1.0,-4.3301)(0.0,-4.3301)
\pspolygon[fillcolor=Apricot](1.5,-5.1962)(2.5,-5.1962)(3.0,-4.3301)(2.0,-4.3301)
\psline[linewidth=0.1,linecolor=white](1.75,-4.7631)(2.75,-4.7631)
\pspolygon[fillcolor=Apricot](2.5,-5.1962)(3.5,-5.1962)(4.0,-4.3301)(3.0,-4.3301)
\psline[linewidth=0.1,linecolor=white](2.75,-4.7631)(3.75,-4.7631)
\pspolygon[fillcolor=Apricot](3.5,-5.1962)(4.5,-5.1962)(5.0,-4.3301)(4.0,-4.3301)
\psline[linewidth=0.1,linecolor=white](3.75,-4.7631)(4.75,-4.7631)
\pspolygon[fillcolor=black](4.5,-5.1962)(5.5,-5.1962)(5.0,-4.3301)(5.0,-4.3301)
\pspolygon[fillcolor=black](-6.0,-6.0622)(-5.0,-6.0622)(-5.5,-5.1962)(-5.5,-5.1962)
\pspolygon[fillcolor=Mahogany](-5.0,-6.0622)(-4.0,-6.0622)(-4.5,-5.1962)(-5.5,-5.1962)
\pspolygon[fillcolor=Mahogany](-4.0,-6.0622)(-3.0,-6.0622)(-3.5,-5.1962)(-4.5,-5.1962)
\pspolygon[fillcolor=Apricot](-3.0,-6.0622)(-2.0,-6.0622)(-1.5,-5.1962)(-2.5,-5.1962)
\psline[linewidth=0.1,linecolor=white](-2.75,-5.6292)(-1.75,-5.6292)
\pspolygon[fillcolor=Apricot](-2.0,-6.0622)(-1.0,-6.0622)(-0.5,-5.1962)(-1.5,-5.1962)
\psline[linewidth=0.1,linecolor=white](-1.75,-5.6292)(-0.75,-5.6292)
\pspolygon[fillcolor=Apricot](-1.0,-6.0622)(0.0,-6.0622)(0.5,-5.1962)(-0.5,-5.1962)
\psline[linewidth=0.1,linecolor=white](-0.75,-5.6292)(0.25,-5.6292)
\pspolygon[fillcolor=Apricot](0.0,-6.0622)(1.0,-6.0622)(1.5,-5.1962)(0.5,-5.1962)
\psline[linewidth=0.1,linecolor=white](0.25,-5.6292)(1.25,-5.6292)
\pspolygon[fillcolor=Apricot](1.0,-6.0622)(2.0,-6.0622)(2.5,-5.1962)(1.5,-5.1962)
\psline[linewidth=0.1,linecolor=white](1.25,-5.6292)(2.25,-5.6292)
\pspolygon[fillcolor=Apricot](2.0,-6.0622)(3.0,-6.0622)(3.5,-5.1962)(2.5,-5.1962)
\psline[linewidth=0.1,linecolor=white](2.25,-5.6292)(3.25,-5.6292)
\pspolygon[fillcolor=Apricot](3.0,-6.0622)(4.0,-6.0622)(4.5,-5.1962)(3.5,-5.1962)
\psline[linewidth=0.1,linecolor=white](3.25,-5.6292)(4.25,-5.6292)
\pspolygon[fillcolor=Apricot](4.0,-6.0622)(5.0,-6.0622)(5.5,-5.1962)(4.5,-5.1962)
\psline[linewidth=0.1,linecolor=white](4.25,-5.6292)(5.25,-5.6292)
\pspolygon[fillcolor=black](5.0,-6.0622)(6.0,-6.0622)(5.5,-5.1962)(5.5,-5.1962)
\end{pspicture}

%% file: graphics/lai-nlp.tex
\psset{unit=0.5cm}
\begin{pspicture}(-1.0,-0.5)(13.0,13.0)
\psset{linewidth=0.03,fillstyle=none,linecolor=gray}
\psline(1.0,1.0)(11.0,1.0)
\psline(1.0,1.0)(1.0,0.0)
\psline(2.0,2.0)(11.0,2.0)
\psline(2.0,2.0)(2.0,0.0)
\psline(3.0,3.0)(11.0,3.0)
\psline(3.0,3.0)(3.0,0.0)
\psline(4.0,4.0)(11.0,4.0)
\psline(4.0,4.0)(4.0,0.0)
\psline(5.0,5.0)(11.0,5.0)
\psline(5.0,5.0)(5.0,0.0)
\psline(6.0,6.0)(11.0,6.0)
\psline(6.0,6.0)(6.0,0.0)
\psline(7.0,7.0)(11.0,7.0)
\psline(7.0,7.0)(7.0,0.0)
\psline(8.0,8.0)(11.0,8.0)
\psline(8.0,8.0)(8.0,0.0)
\psline(9.0,9.0)(11.0,9.0)
\psline(9.0,9.0)(9.0,0.0)
\psline(10.0,10.0)(11.0,10.0)
\psline(10.0,10.0)(10.0,0.0)
\psline(11.0,11.0)(11.0,11.0)
\psline(11.0,11.0)(11.0,0.0)
\psset{linewidth=0.1,fillstyle=none,linecolor=black}
\psline{->}(0,0)(12,0)
\psline{->}(0,0)(0,12)
\psline[linestyle=dotted,linecolor=gray](0,0)(12,12)
\rput(12.5,0){{\footnotesize $x$}}
\rput(0,12.5){{\footnotesize $y$}}
\rput(12.25,12.25){{\footnotesize $y=x$}}
\psset{linewidth=0.225,fillstyle=none,linecolor=blue}
\psline(5.0,5.0)(5.0,4.0)(5.0,3.0)(6.0,3.0)(6.0,2.0)(6.0,1.0)(6.0,0.0)(7.0,0.0)
\psline(7.0,7.0)(8.0,7.0)(8.0,6.0)(8.0,5.0)(8.0,4.0)(9.0,4.0)(9.0,3.0)(10.0,3.0)(10.0,2.0)(10.0,1.0)(10.0,0.0)
\psline(9.0,9.0)(9.0,8.0)(10.0,8.0)(11.0,8.0)(11.0,7.0)(11.0,6.0)(11.0,5.0)(11.0,4.0)(11.0,3.0)(11.0,2.0)(11.0,1.0)(11.0,0.0)
\rput(0.5,-0.3){{\tiny\gray $0$}}
\rput(1.5,-0.3){{\tiny\gray $1$}}
\rput(1.5,0.7){{\tiny\gray $-1$}}
\rput(2.5,-0.3){{\tiny\gray $2$}}
\rput(2.5,0.7){{\tiny\gray $0$}}
\rput(2.5,1.7){{\tiny\gray $-2$}}
\rput(3.5,-0.3){{\tiny\gray $3$}}
\rput(3.5,0.7){{\tiny\gray $1$}}
\rput(3.5,1.7){{\tiny\gray $-1$}}
\rput(3.5,2.7){{\tiny\gray $-3$}}
\rput(4.5,-0.3){{\tiny\gray $4$}}
\rput(4.5,0.7){{\tiny\gray $2$}}
\rput(4.5,1.7){{\tiny\gray $0$}}
\rput(4.5,2.7){{\tiny\gray $-2$}}
\rput(4.5,3.7){{\tiny\gray $-4$}}
\rput(5.5,-0.3){{\tiny\gray $5$}}
\rput(5.5,0.7){{\tiny\gray $3$}}
\rput(5.5,1.7){{\tiny\gray $1$}}
\rput(5.5,2.7){{\tiny\gray $-1$}}
\rput(5.5,3.7){{\tiny\gray $-3$}}
\rput(5.5,4.7){{\tiny\gray $-5$}}
\rput(6.5,-0.3){{\tiny\gray $6$}}
\rput(6.5,0.7){{\tiny\gray $4$}}
\rput(6.5,1.7){{\tiny\gray $2$}}
\rput(6.5,2.7){{\tiny\gray $0$}}
\rput(6.5,3.7){{\tiny\gray $-2$}}
\rput(6.5,4.7){{\tiny\gray $-4$}}
\rput(6.5,5.7){{\tiny\gray $-6$}}
\rput(7.5,-0.3){{\tiny\gray $7$}}
\rput(7.5,0.7){{\tiny\gray $5$}}
\rput(7.5,1.7){{\tiny\gray $3$}}
\rput(7.5,2.7){{\tiny\gray $1$}}
\rput(7.5,3.7){{\tiny\gray $-1$}}
\rput(7.5,4.7){{\tiny\gray $-3$}}
\rput(7.5,5.7){{\tiny\gray $-5$}}
\rput(7.5,6.7){{\tiny\gray $-7$}}
\rput(8.5,-0.3){{\tiny\gray $8$}}
\rput(8.5,0.7){{\tiny\gray $6$}}
\rput(8.5,1.7){{\tiny\gray $4$}}
\rput(8.5,2.7){{\tiny\gray $2$}}
\rput(8.5,3.7){{\tiny\gray $0$}}
\rput(8.5,4.7){{\tiny\gray $-2$}}
\rput(8.5,5.7){{\tiny\gray $-4$}}
\rput(8.5,6.7){{\tiny\gray $-6$}}
\rput(8.5,7.7){{\tiny\gray $-8$}}
\rput(9.5,-0.3){{\tiny\gray $9$}}
\rput(9.5,0.7){{\tiny\gray $7$}}
\rput(9.5,1.7){{\tiny\gray $5$}}
\rput(9.5,2.7){{\tiny\gray $3$}}
\rput(9.5,3.7){{\tiny\gray $1$}}
\rput(9.5,4.7){{\tiny\gray $-1$}}
\rput(9.5,5.7){{\tiny\gray $-3$}}
\rput(9.5,6.7){{\tiny\gray $-5$}}
\rput(9.5,7.7){{\tiny\gray $-7$}}
\rput(9.5,8.7){{\tiny\gray $-9$}}
\rput(10.5,-0.3){{\tiny\gray $10$}}
\rput(10.5,0.7){{\tiny\gray $8$}}
\rput(10.5,1.7){{\tiny\gray $6$}}
\rput(10.5,2.7){{\tiny\gray $4$}}
\rput(10.5,3.7){{\tiny\gray $2$}}
\rput(10.5,4.7){{\tiny\gray $0$}}
\rput(10.5,5.7){{\tiny\gray $-2$}}
\rput(10.5,6.7){{\tiny\gray $-4$}}
\rput(10.5,7.7){{\tiny\gray $-6$}}
\rput(10.5,8.7){{\tiny\gray $-8$}}
\rput(10.5,9.7){{\tiny\gray $-10$}}
\end{pspicture}

%% file: graphics/lai-hex-quarter.tex
\psset{unit=0.5cm}
\begin{pspicture}(-1.0,-10.8923)(8.0,3.4641)
\psset{linewidth=0.02,fillstyle=none,linecolor=black}
\psset{fillstyle=solid,linecolor=black}
\pspolygon[fillcolor=Apricot](0.0,-0.866)(1.0,-0.866)(1.5,0.0)(0.5,0.0)
\psline[linewidth=0.1,linecolor=white](0.25,-0.433)(1.25,-0.433)
\pspolygon[fillcolor=Apricot](1.0,-0.866)(2.0,-0.866)(2.5,0.0)(1.5,0.0)
\psline[linewidth=0.1,linecolor=white](1.25,-0.433)(2.25,-0.433)
\pspolygon[fillcolor=Tan](2.0,-0.866)(2.5,-1.7321)(3.0,-0.866)(2.5,0.0)
\psline[linewidth=0.1,linecolor=white](2.25,-0.433)(2.75,-1.299)
\rput(2.5,-0.866){{\tiny\black $4$}}
\pspolygon[fillcolor=Mahogany](3.0,-0.866)(4.0,-0.866)(3.5,0.0)(2.5,0.0)
\pspolygon[fillcolor=Mahogany](4.0,-0.866)(5.0,-0.866)(4.5,0.0)(3.5,0.0)
\pspolygon[fillcolor=Mahogany](0.5,-1.7321)(1.5,-1.7321)(1.0,-0.866)(0.0,-0.866)
\pspolygon[fillcolor=Mahogany](1.5,-1.7321)(2.5,-1.7321)(2.0,-0.866)(1.0,-0.866)
\pspolygon[fillcolor=Apricot](2.5,-1.7321)(3.5,-1.7321)(4.0,-0.866)(3.0,-0.866)
\psline[linewidth=0.1,linecolor=white](2.75,-1.299)(3.75,-1.299)
\pspolygon[fillcolor=Apricot](3.5,-1.7321)(4.5,-1.7321)(5.0,-0.866)(4.0,-0.866)
\psline[linewidth=0.1,linecolor=white](3.75,-1.299)(4.75,-1.299)
\pspolygon[fillcolor=black](4.5,-1.7321)(5.5,-1.7321)(5.0,-0.866)(5.0,-0.866)
\pspolygon[fillcolor=Tan](0.0,-2.5981)(0.5,-3.4641)(1.0,-2.5981)(0.5,-1.7321)
\psline[linewidth=0.1,linecolor=white](0.25,-2.1651)(0.75,-3.0311)
\rput(0.5,-2.5981){{\tiny\black $0$}}
\pspolygon[fillcolor=Mahogany](1.0,-2.5981)(2.0,-2.5981)(1.5,-1.7321)(0.5,-1.7321)
\pspolygon[fillcolor=Mahogany](2.0,-2.5981)(3.0,-2.5981)(2.5,-1.7321)(1.5,-1.7321)
\pspolygon[fillcolor=Mahogany](3.0,-2.5981)(4.0,-2.5981)(3.5,-1.7321)(2.5,-1.7321)
\pspolygon[fillcolor=Mahogany](4.0,-2.5981)(5.0,-2.5981)(4.5,-1.7321)(3.5,-1.7321)
\pspolygon[fillcolor=Mahogany](5.0,-2.5981)(6.0,-2.5981)(5.5,-1.7321)(4.5,-1.7321)
\pspolygon[fillcolor=Apricot](0.5,-3.4641)(1.5,-3.4641)(2.0,-2.5981)(1.0,-2.5981)
\psline[linewidth=0.1,linecolor=white](0.75,-3.0311)(1.75,-3.0311)
\pspolygon[fillcolor=Apricot](1.5,-3.4641)(2.5,-3.4641)(3.0,-2.5981)(2.0,-2.5981)
\psline[linewidth=0.1,linecolor=white](1.75,-3.0311)(2.75,-3.0311)
\pspolygon[fillcolor=Apricot](2.5,-3.4641)(3.5,-3.4641)(4.0,-2.5981)(3.0,-2.5981)
\psline[linewidth=0.1,linecolor=white](2.75,-3.0311)(3.75,-3.0311)
\pspolygon[fillcolor=Apricot](3.5,-3.4641)(4.5,-3.4641)(5.0,-2.5981)(4.0,-2.5981)
\psline[linewidth=0.1,linecolor=white](3.75,-3.0311)(4.75,-3.0311)
\pspolygon[fillcolor=Apricot](4.5,-3.4641)(5.5,-3.4641)(6.0,-2.5981)(5.0,-2.5981)
\psline[linewidth=0.1,linecolor=white](4.75,-3.0311)(5.75,-3.0311)
\pspolygon[fillcolor=black](5.5,-3.4641)(6.5,-3.4641)(6.0,-2.5981)(6.0,-2.5981)
\pspolygon[fillcolor=Apricot](0.0,-4.3301)(1.0,-4.3301)(1.5,-3.4641)(0.5,-3.4641)
\psline[linewidth=0.1,linecolor=white](0.25,-3.8971)(1.25,-3.8971)
\pspolygon[fillcolor=Tan](1.0,-4.3301)(1.5,-5.1962)(2.0,-4.3301)(1.5,-3.4641)
\psline[linewidth=0.1,linecolor=white](1.25,-3.8971)(1.75,-4.7631)
\rput(1.5,-4.3301){{\tiny\black $2$}}
\pspolygon[fillcolor=Mahogany](2.0,-4.3301)(3.0,-4.3301)(2.5,-3.4641)(1.5,-3.4641)
\pspolygon[fillcolor=Mahogany](3.0,-4.3301)(4.0,-4.3301)(3.5,-3.4641)(2.5,-3.4641)
\pspolygon[fillcolor=Mahogany](4.0,-4.3301)(5.0,-4.3301)(4.5,-3.4641)(3.5,-3.4641)
\pspolygon[fillcolor=Mahogany](5.0,-4.3301)(6.0,-4.3301)(5.5,-3.4641)(4.5,-3.4641)
\pspolygon[fillcolor=Mahogany](6.0,-4.3301)(7.0,-4.3301)(6.5,-3.4641)(5.5,-3.4641)
\pspolygon[fillcolor=Mahogany](0.5,-5.1962)(1.5,-5.1962)(1.0,-4.3301)(0.0,-4.3301)
\pspolygon[fillcolor=Tan](1.5,-5.1962)(2.0,-6.0622)(2.5,-5.1962)(2.0,-4.3301)
\psline[linewidth=0.1,linecolor=white](1.75,-4.7631)(2.25,-5.6292)
\rput(2.0,-5.1962){{\tiny\black $3$}}
\pspolygon[fillcolor=Mahogany](2.5,-5.1962)(3.5,-5.1962)(3.0,-4.3301)(2.0,-4.3301)
\pspolygon[fillcolor=Mahogany](3.5,-5.1962)(4.5,-5.1962)(4.0,-4.3301)(3.0,-4.3301)
\pspolygon[fillcolor=Mahogany](4.5,-5.1962)(5.5,-5.1962)(5.0,-4.3301)(4.0,-4.3301)
\pspolygon[fillcolor=Mahogany](5.5,-5.1962)(6.5,-5.1962)(6.0,-4.3301)(5.0,-4.3301)
\pspolygon[fillcolor=Mahogany](6.5,-5.1962)(7.5,-5.1962)(7.0,-4.3301)(6.0,-4.3301)
\pspolygon[fillcolor=Apricot](0.0,-6.0622)(1.0,-6.0622)(1.5,-5.1962)(0.5,-5.1962)
\psline[linewidth=0.1,linecolor=white](0.25,-5.6292)(1.25,-5.6292)
\pspolygon[fillcolor=Tan](1.0,-6.0622)(1.5,-6.9282)(2.0,-6.0622)(1.5,-5.1962)
\psline[linewidth=0.1,linecolor=white](1.25,-5.6292)(1.75,-6.4952)
\rput(1.5,-6.0622){{\tiny\black $2$}}
\pspolygon[fillcolor=Apricot](2.0,-6.0622)(3.0,-6.0622)(3.5,-5.1962)(2.5,-5.1962)
\psline[linewidth=0.1,linecolor=white](2.25,-5.6292)(3.25,-5.6292)
\pspolygon[fillcolor=Apricot](3.0,-6.0622)(4.0,-6.0622)(4.5,-5.1962)(3.5,-5.1962)
\psline[linewidth=0.1,linecolor=white](3.25,-5.6292)(4.25,-5.6292)
\pspolygon[fillcolor=Apricot](4.0,-6.0622)(5.0,-6.0622)(5.5,-5.1962)(4.5,-5.1962)
\psline[linewidth=0.1,linecolor=white](4.25,-5.6292)(5.25,-5.6292)
\pspolygon[fillcolor=Apricot](5.0,-6.0622)(6.0,-6.0622)(6.5,-5.1962)(5.5,-5.1962)
\psline[linewidth=0.1,linecolor=white](5.25,-5.6292)(6.25,-5.6292)
\pspolygon[fillcolor=Apricot](6.0,-6.0622)(7.0,-6.0622)(7.5,-5.1962)(6.5,-5.1962)
\psline[linewidth=0.1,linecolor=white](6.25,-5.6292)(7.25,-5.6292)
\pspolygon[fillcolor=black](7.0,-6.0622)(8.0,-6.0622)(7.5,-5.1962)(7.5,-5.1962)
\pspolygon[fillcolor=Mahogany](0.5,-6.9282)(1.5,-6.9282)(1.0,-6.0622)(0.0,-6.0622)
\pspolygon[fillcolor=Apricot](1.5,-6.9282)(2.5,-6.9282)(3.0,-6.0622)(2.0,-6.0622)
\psline[linewidth=0.1,linecolor=white](1.75,-6.4952)(2.75,-6.4952)
\pspolygon[fillcolor=Apricot](2.5,-6.9282)(3.5,-6.9282)(4.0,-6.0622)(3.0,-6.0622)
\psline[linewidth=0.1,linecolor=white](2.75,-6.4952)(3.75,-6.4952)
\pspolygon[fillcolor=Apricot](3.5,-6.9282)(4.5,-6.9282)(5.0,-6.0622)(4.0,-6.0622)
\psline[linewidth=0.1,linecolor=white](3.75,-6.4952)(4.75,-6.4952)
\pspolygon[fillcolor=Tan](4.5,-6.9282)(5.0,-7.7942)(5.5,-6.9282)(5.0,-6.0622)
\psline[linewidth=0.1,linecolor=white](4.75,-6.4952)(5.25,-7.3612)
\rput(5.0,-6.9282){{\tiny\black $9$}}
\pspolygon[fillcolor=Mahogany](5.5,-6.9282)(6.5,-6.9282)(6.0,-6.0622)(5.0,-6.0622)
\pspolygon[fillcolor=Mahogany](6.5,-6.9282)(7.5,-6.9282)(7.0,-6.0622)(6.0,-6.0622)
\pspolygon[fillcolor=Mahogany](7.5,-6.9282)(8.5,-6.9282)(8.0,-6.0622)(7.0,-6.0622)
\pspolygon[fillcolor=Tan](0.0,-7.7942)(0.5,-8.6603)(1.0,-7.7942)(0.5,-6.9282)
\psline[linewidth=0.1,linecolor=white](0.25,-7.3612)(0.75,-8.2272)
\rput(0.5,-7.7942){{\tiny\black $0$}}
\pspolygon[fillcolor=Mahogany](1.0,-7.7942)(2.0,-7.7942)(1.5,-6.9282)(0.5,-6.9282)
\pspolygon[fillcolor=Mahogany](2.0,-7.7942)(3.0,-7.7942)(2.5,-6.9282)(1.5,-6.9282)
\pspolygon[fillcolor=Mahogany](3.0,-7.7942)(4.0,-7.7942)(3.5,-6.9282)(2.5,-6.9282)
\pspolygon[fillcolor=Mahogany](4.0,-7.7942)(5.0,-7.7942)(4.5,-6.9282)(3.5,-6.9282)
\pspolygon[fillcolor=Apricot](5.0,-7.7942)(6.0,-7.7942)(6.5,-6.9282)(5.5,-6.9282)
\psline[linewidth=0.1,linecolor=white](5.25,-7.3612)(6.25,-7.3612)
\pspolygon[fillcolor=Apricot](6.0,-7.7942)(7.0,-7.7942)(7.5,-6.9282)(6.5,-6.9282)
\psline[linewidth=0.1,linecolor=white](6.25,-7.3612)(7.25,-7.3612)
\pspolygon[fillcolor=Tan](7.0,-7.7942)(7.5,-8.6603)(8.0,-7.7942)(7.5,-6.9282)
\psline[linewidth=0.1,linecolor=white](7.25,-7.3612)(7.75,-8.2272)
\rput(7.5,-7.7942){{\tiny\black $14$}}
\pspolygon[fillcolor=Mahogany](8.0,-7.7942)(9.0,-7.7942)(8.5,-6.9282)(7.5,-6.9282)
\pspolygon[fillcolor=Apricot](0.5,-8.6603)(1.5,-8.6603)(2.0,-7.7942)(1.0,-7.7942)
\psline[linewidth=0.1,linecolor=white](0.75,-8.2272)(1.75,-8.2272)
\pspolygon[fillcolor=Apricot](1.5,-8.6603)(2.5,-8.6603)(3.0,-7.7942)(2.0,-7.7942)
\psline[linewidth=0.1,linecolor=white](1.75,-8.2272)(2.75,-8.2272)
\pspolygon[fillcolor=Apricot](2.5,-8.6603)(3.5,-8.6603)(4.0,-7.7942)(3.0,-7.7942)
\psline[linewidth=0.1,linecolor=white](2.75,-8.2272)(3.75,-8.2272)
\pspolygon[fillcolor=Tan](3.5,-8.6603)(4.0,-9.5263)(4.5,-8.6603)(4.0,-7.7942)
\psline[linewidth=0.1,linecolor=white](3.75,-8.2272)(4.25,-9.0933)
\rput(4.0,-8.6603){{\tiny\black $7$}}
\pspolygon[fillcolor=Mahogany](4.5,-8.6603)(5.5,-8.6603)(5.0,-7.7942)(4.0,-7.7942)
\pspolygon[fillcolor=Mahogany](5.5,-8.6603)(6.5,-8.6603)(6.0,-7.7942)(5.0,-7.7942)
\pspolygon[fillcolor=Mahogany](6.5,-8.6603)(7.5,-8.6603)(7.0,-7.7942)(6.0,-7.7942)
\pspolygon[fillcolor=Apricot](7.5,-8.6603)(8.5,-8.6603)(9.0,-7.7942)(8.0,-7.7942)
\psline[linewidth=0.1,linecolor=white](7.75,-8.2272)(8.75,-8.2272)
\pspolygon[fillcolor=black](8.5,-8.6603)(9.5,-8.6603)(9.0,-7.7942)(9.0,-7.7942)
\pspolygon[fillcolor=Tan](0.0,-9.5263)(0.5,-10.3923)(1.0,-9.5263)(0.5,-8.6603)
\psline[linewidth=0.1,linecolor=white](0.25,-9.0933)(0.75,-9.9593)
\rput(0.5,-9.5263){{\tiny\black $0$}}
\pspolygon[fillcolor=Mahogany](1.0,-9.5263)(2.0,-9.5263)(1.5,-8.6603)(0.5,-8.6603)
\pspolygon[fillcolor=Mahogany](2.0,-9.5263)(3.0,-9.5263)(2.5,-8.6603)(1.5,-8.6603)
\pspolygon[fillcolor=Mahogany](3.0,-9.5263)(4.0,-9.5263)(3.5,-8.6603)(2.5,-8.6603)
\pspolygon[fillcolor=Apricot](4.0,-9.5263)(5.0,-9.5263)(5.5,-8.6603)(4.5,-8.6603)
\psline[linewidth=0.1,linecolor=white](4.25,-9.0933)(5.25,-9.0933)
\pspolygon[fillcolor=Apricot](5.0,-9.5263)(6.0,-9.5263)(6.5,-8.6603)(5.5,-8.6603)
\psline[linewidth=0.1,linecolor=white](5.25,-9.0933)(6.25,-9.0933)
\pspolygon[fillcolor=Apricot](6.0,-9.5263)(7.0,-9.5263)(7.5,-8.6603)(6.5,-8.6603)
\psline[linewidth=0.1,linecolor=white](6.25,-9.0933)(7.25,-9.0933)
\pspolygon[fillcolor=Apricot](7.0,-9.5263)(8.0,-9.5263)(8.5,-8.6603)(7.5,-8.6603)
\psline[linewidth=0.1,linecolor=white](7.25,-9.0933)(8.25,-9.0933)
\pspolygon[fillcolor=Apricot](8.0,-9.5263)(9.0,-9.5263)(9.5,-8.6603)(8.5,-8.6603)
\psline[linewidth=0.1,linecolor=white](8.25,-9.0933)(9.25,-9.0933)
\pspolygon[fillcolor=black](9.0,-9.5263)(10.0,-9.5263)(9.5,-8.6603)(9.5,-8.6603)
\pspolygon[fillcolor=Apricot](0.5,-10.3923)(1.5,-10.3923)(2.0,-9.5263)(1.0,-9.5263)
\psline[linewidth=0.1,linecolor=white](0.75,-9.9593)(1.75,-9.9593)
\pspolygon[fillcolor=Apricot](1.5,-10.3923)(2.5,-10.3923)(3.0,-9.5263)(2.0,-9.5263)
\psline[linewidth=0.1,linecolor=white](1.75,-9.9593)(2.75,-9.9593)
\pspolygon[fillcolor=Apricot](2.5,-10.3923)(3.5,-10.3923)(4.0,-9.5263)(3.0,-9.5263)
\psline[linewidth=0.1,linecolor=white](2.75,-9.9593)(3.75,-9.9593)
\pspolygon[fillcolor=Apricot](3.5,-10.3923)(4.5,-10.3923)(5.0,-9.5263)(4.0,-9.5263)
\psline[linewidth=0.1,linecolor=white](3.75,-9.9593)(4.75,-9.9593)
\pspolygon[fillcolor=Apricot](4.5,-10.3923)(5.5,-10.3923)(6.0,-9.5263)(5.0,-9.5263)
\psline[linewidth=0.1,linecolor=white](4.75,-9.9593)(5.75,-9.9593)
\pspolygon[fillcolor=Apricot](5.5,-10.3923)(6.5,-10.3923)(7.0,-9.5263)(6.0,-9.5263)
\psline[linewidth=0.1,linecolor=white](5.75,-9.9593)(6.75,-9.9593)
\pspolygon[fillcolor=Apricot](6.5,-10.3923)(7.5,-10.3923)(8.0,-9.5263)(7.0,-9.5263)
\psline[linewidth=0.1,linecolor=white](6.75,-9.9593)(7.75,-9.9593)
\pspolygon[fillcolor=Apricot](7.5,-10.3923)(8.5,-10.3923)(9.0,-9.5263)(8.0,-9.5263)
\psline[linewidth=0.1,linecolor=white](7.75,-9.9593)(8.75,-9.9593)
\pspolygon[fillcolor=Apricot](8.5,-10.3923)(9.5,-10.3923)(10.0,-9.5263)(9.0,-9.5263)
\psline[linewidth=0.1,linecolor=white](8.75,-9.9593)(9.75,-9.9593)
\pspolygon[fillcolor=black](9.5,-10.3923)(10.5,-10.3923)(10.0,-9.5263)(10.0,-9.5263)
\end{pspicture}

%% file: graphics/lai-nlp-1-5-b.tex
\psset{unit=0.5cm}
\begin{pspicture}(-1.0,-0.5)(22.0,12.0)
\psset{linewidth=0.03,fillstyle=none,linecolor=gray}
\psline(1.0,1.0)(20.0,1.0)
\psline(1.0,1.0)(1.0,0.0)
\psline(2.0,2.0)(20.0,2.0)
\psline(2.0,2.0)(2.0,0.0)
\psline(3.0,3.0)(20.0,3.0)
\psline(3.0,3.0)(3.0,0.0)
\psline(4.0,4.0)(20.0,4.0)
\psline(4.0,4.0)(4.0,0.0)
\psline(5.0,5.0)(20.0,5.0)
\psline(5.0,5.0)(5.0,0.0)
\psline(6.0,6.0)(20.0,6.0)
\psline(6.0,6.0)(6.0,0.0)
\psline(7.0,7.0)(20.0,7.0)
\psline(7.0,7.0)(7.0,0.0)
\psline(8.0,8.0)(20.0,8.0)
\psline(8.0,8.0)(8.0,0.0)
\psline(9.0,9.0)(20.0,9.0)
\psline(9.0,9.0)(9.0,0.0)
\psline(10.0,10.0)(20.0,10.0)
\psline(10.0,10.0)(10.0,0.0)
\psline(11.0,10.33)(11.0,0.0)
\psline(12.0,10.33)(12.0,0.0)
\psline(13.0,10.33)(13.0,0.0)
\psline(14.0,10.33)(14.0,0.0)
\psline(15.0,10.33)(15.0,0.0)
\psline(16.0,10.33)(16.0,0.0)
\psline(17.0,10.33)(17.0,0.0)
\psline(18.0,10.33)(18.0,0.0)
\psline(19.0,10.33)(19.0,0.0)
\psset{linewidth=0.1,fillstyle=none,linecolor=black}
\psline{->}(0,0)(21,0)
\psline{->}(0,0)(0,11)
\psline[linestyle=dotted,linecolor=gray](0,0)(11,11)
\psline[linestyle=dashed,linecolor=red](0,0)(21,10.5)
\rput(21.5,0){$x$}\rput(0,11.5){$y$}\rput(11.5,11.5){$y=x$}
\rput(19.25,10.75){{\red $2y=x$}}
\psset{linewidth=0.225,fillstyle=none,linecolor=blue}
\rput(1,0){\psline(8.0,4.0)(8.0,3.0)(8.0,2.0)(9.0,2.0)(9.0,1.0)(9.0,0.0)}
\rput(1,0){\psline(10.0,5.0)(11.0,5.0)(11.0,4.0)(11.0,3.0)(11.0,2.0)(11.0,1.0)(11.0,0.0)}
\rput(1,0){\psline(12.0,6.0)(12.0,5.0)(13.0,5.0)(14.0,5.0)(14.0,4.0)(14.0,3.0)(14.0,2.0)(14.0,1.0)(14.0,0.0)}
\rput(1,0){\psline(14.0,7.0)(14.0,6.0)(15.0,6.0)(15.0,5.0)(15.0,4.0)(15.0,3.0)(16.0,3.0)(16.0,2.0)(16.0,1.0)(17.0,1.0)(17.0,0.0)}
\rput(1,0){\psline(16.0,8.0)(17.0,8.0)(17.0,7.0)(17.0,6.0)(17.0,5.0)(18.0,5.0)(18.0,4.0)(18.0,3.0)(18.0,2.0)(18.0,1.0)(18.0,0.0)}
\rput(1,0){\psline(18.0,9.0)(19.0,9.0)(19.0,8.0)(19.0,7.0)(19.0,6.0)(19.0,5.0)(19.0,4.0)(19.0,3.0)(19.0,2.0)(19.0,1.0)(19.0,0.0)}
\rput(0.5,-0.3){{\tiny\gray $0$}}
\rput(1.5,-0.3){{\tiny\gray $1$}}
\rput(1.5,0.7){{\tiny\gray $-1$}}
\rput(2.5,-0.3){{\tiny\gray $2$}}
\rput(2.5,0.7){{\tiny\gray $0$}}
\rput(2.5,1.7){{\tiny\gray $-2$}}
\rput(3.5,-0.3){{\tiny\gray $3$}}
\rput(3.5,0.7){{\tiny\gray $1$}}
\rput(3.5,1.7){{\tiny\gray $-1$}}
\rput(3.5,2.7){{\tiny\gray $-3$}}
\rput(4.5,-0.3){{\tiny\gray $4$}}
\rput(4.5,0.7){{\tiny\gray $2$}}
\rput(4.5,1.7){{\tiny\gray $0$}}
\rput(4.5,2.7){{\tiny\gray $-2$}}
\rput(4.5,3.7){{\tiny\gray $-4$}}
\rput(5.5,-0.3){{\tiny\gray $5$}}
\rput(5.5,0.7){{\tiny\gray $3$}}
\rput(5.5,1.7){{\tiny\gray $1$}}
\rput(5.5,2.7){{\tiny\gray $-1$}}
\rput(5.5,3.7){{\tiny\gray $-3$}}
\rput(5.5,4.7){{\tiny\gray $-5$}}
\rput(6.5,-0.3){{\tiny\gray $6$}}
\rput(6.5,0.7){{\tiny\gray $4$}}
\rput(6.5,1.7){{\tiny\gray $2$}}
\rput(6.5,2.7){{\tiny\gray $0$}}
\rput(6.5,3.7){{\tiny\gray $-2$}}
\rput(6.5,4.7){{\tiny\gray $-4$}}
\rput(6.5,5.7){{\tiny\gray $-6$}}
\rput(7.5,-0.3){{\tiny\gray $7$}}
\rput(7.5,0.7){{\tiny\gray $5$}}
\rput(7.5,1.7){{\tiny\gray $3$}}
\rput(7.5,2.7){{\tiny\gray $1$}}
\rput(7.5,3.7){{\tiny\gray $-1$}}
\rput(7.5,4.7){{\tiny\gray $-3$}}
\rput(7.5,5.7){{\tiny\gray $-5$}}
\rput(7.5,6.7){{\tiny\gray $-7$}}
\rput(8.5,-0.3){{\tiny\gray $8$}}
\rput(8.5,0.7){{\tiny\gray $6$}}
\rput(8.5,1.7){{\tiny\gray $4$}}
\rput(8.5,2.7){{\tiny\gray $2$}}
\rput(8.5,3.7){{\tiny\gray $0$}}
\rput(8.5,4.7){{\tiny\gray $-2$}}
\rput(8.5,5.7){{\tiny\gray $-4$}}
\rput(8.5,6.7){{\tiny\gray $-6$}}
\rput(8.5,7.7){{\tiny\gray $-8$}}
\rput(9.5,-0.3){{\tiny\gray $9$}}
\rput(9.5,0.7){{\tiny\gray $7$}}
\rput(9.5,1.7){{\tiny\gray $5$}}
\rput(9.5,2.7){{\tiny\gray $3$}}
\rput(9.5,3.7){{\tiny\gray $1$}}
\rput(9.5,4.7){{\tiny\gray $-1$}}
\rput(9.5,5.7){{\tiny\gray $-3$}}
\rput(9.5,6.7){{\tiny\gray $-5$}}
\rput(9.5,7.7){{\tiny\gray $-7$}}
\rput(9.5,8.7){{\tiny\gray $-9$}}
\rput(10.5,-0.3){{\tiny\gray $10$}}
\rput(10.5,0.7){{\tiny\gray $8$}}
\rput(10.5,1.7){{\tiny\gray $6$}}
\rput(10.5,2.7){{\tiny\gray $4$}}
\rput(10.5,3.7){{\tiny\gray $2$}}
\rput(10.5,4.7){{\tiny\gray $0$}}
\rput(10.5,5.7){{\tiny\gray $-2$}}
\rput(10.5,6.7){{\tiny\gray $-4$}}
\rput(10.5,7.7){{\tiny\gray $-6$}}
\rput(10.5,8.7){{\tiny\gray $-8$}}
\rput(10.5,9.7){{\tiny\gray $-10$}}
\rput(11.5,-0.3){{\tiny\gray $11$}}
\rput(11.5,0.7){{\tiny\gray $9$}}
\rput(11.5,1.7){{\tiny\gray $7$}}
\rput(11.5,2.7){{\tiny\gray $5$}}
\rput(11.5,3.7){{\tiny\gray $3$}}
\rput(11.5,4.7){{\tiny\gray $1$}}
\rput(11.5,5.7){{\tiny\gray $-1$}}
\rput(11.5,6.7){{\tiny\gray $-3$}}
\rput(11.5,7.7){{\tiny\gray $-5$}}
\rput(11.5,8.7){{\tiny\gray $-7$}}
\rput(11.5,9.7){{\tiny\gray $-9$}}
\rput(12.5,-0.3){{\tiny\gray $12$}}
\rput(12.5,0.7){{\tiny\gray $10$}}
\rput(12.5,1.7){{\tiny\gray $8$}}
\rput(12.5,2.7){{\tiny\gray $6$}}
\rput(12.5,3.7){{\tiny\gray $4$}}
\rput(12.5,4.7){{\tiny\gray $2$}}
\rput(12.5,5.7){{\tiny\gray $0$}}
\rput(12.5,6.7){{\tiny\gray $-2$}}
\rput(12.5,7.7){{\tiny\gray $-4$}}
\rput(12.5,8.7){{\tiny\gray $-6$}}
\rput(12.5,9.7){{\tiny\gray $-8$}}
\rput(13.5,-0.3){{\tiny\gray $13$}}
\rput(13.5,0.7){{\tiny\gray $11$}}
\rput(13.5,1.7){{\tiny\gray $9$}}
\rput(13.5,2.7){{\tiny\gray $7$}}
\rput(13.5,3.7){{\tiny\gray $5$}}
\rput(13.5,4.7){{\tiny\gray $3$}}
\rput(13.5,5.7){{\tiny\gray $1$}}
\rput(13.5,6.7){{\tiny\gray $-1$}}
\rput(13.5,7.7){{\tiny\gray $-3$}}
\rput(13.5,8.7){{\tiny\gray $-5$}}
\rput(13.5,9.7){{\tiny\gray $-7$}}
\rput(14.5,-0.3){{\tiny\gray $14$}}
\rput(14.5,0.7){{\tiny\gray $12$}}
\rput(14.5,1.7){{\tiny\gray $10$}}
\rput(14.5,2.7){{\tiny\gray $8$}}
\rput(14.5,3.7){{\tiny\gray $6$}}
\rput(14.5,4.7){{\tiny\gray $4$}}
\rput(14.5,5.7){{\tiny\gray $2$}}
\rput(14.5,6.7){{\tiny\gray $0$}}
\rput(14.5,7.7){{\tiny\gray $-2$}}
\rput(14.5,8.7){{\tiny\gray $-4$}}
\rput(14.5,9.7){{\tiny\gray $-6$}}
\rput(15.5,-0.3){{\tiny\gray $15$}}
\rput(15.5,0.7){{\tiny\gray $13$}}
\rput(15.5,1.7){{\tiny\gray $11$}}
\rput(15.5,2.7){{\tiny\gray $9$}}
\rput(15.5,3.7){{\tiny\gray $7$}}
\rput(15.5,4.7){{\tiny\gray $5$}}
\rput(15.5,5.7){{\tiny\gray $3$}}
\rput(15.5,6.7){{\tiny\gray $1$}}
\rput(15.5,7.7){{\tiny\gray $-1$}}
\rput(15.5,8.7){{\tiny\gray $-3$}}
\rput(15.5,9.7){{\tiny\gray $-5$}}
\rput(16.5,-0.3){{\tiny\gray $16$}}
\rput(16.5,0.7){{\tiny\gray $14$}}
\rput(16.5,1.7){{\tiny\gray $12$}}
\rput(16.5,2.7){{\tiny\gray $10$}}
\rput(16.5,3.7){{\tiny\gray $8$}}
\rput(16.5,4.7){{\tiny\gray $6$}}
\rput(16.5,5.7){{\tiny\gray $4$}}
\rput(16.5,6.7){{\tiny\gray $2$}}
\rput(16.5,7.7){{\tiny\gray $0$}}
\rput(16.5,8.7){{\tiny\gray $-2$}}
\rput(16.5,9.7){{\tiny\gray $-4$}}
\rput(17.5,-0.3){{\tiny\gray $17$}}
\rput(17.5,0.7){{\tiny\gray $15$}}
\rput(17.5,1.7){{\tiny\gray $13$}}
\rput(17.5,2.7){{\tiny\gray $11$}}
\rput(17.5,3.7){{\tiny\gray $9$}}
\rput(17.5,4.7){{\tiny\gray $7$}}
\rput(17.5,5.7){{\tiny\gray $5$}}
\rput(17.5,6.7){{\tiny\gray $3$}}
\rput(17.5,7.7){{\tiny\gray $1$}}
\rput(17.5,8.7){{\tiny\gray $-1$}}
\rput(17.5,9.7){{\tiny\gray $-3$}}
\rput(18.5,-0.3){{\tiny\gray $18$}}
\rput(18.5,0.7){{\tiny\gray $16$}}
\rput(18.5,1.7){{\tiny\gray $14$}}
\rput(18.5,2.7){{\tiny\gray $12$}}
\rput(18.5,3.7){{\tiny\gray $10$}}
\rput(18.5,4.7){{\tiny\gray $8$}}
\rput(18.5,5.7){{\tiny\gray $6$}}
\rput(18.5,6.7){{\tiny\gray $4$}}
\rput(18.5,7.7){{\tiny\gray $2$}}
\rput(18.5,8.7){{\tiny\gray $0$}}
\rput(18.5,9.7){{\tiny\gray $-2$}}
\rput(19.5,-0.3){{\tiny\gray $19$}}
\rput(19.5,0.7){{\tiny\gray $17$}}
\rput(19.5,1.7){{\tiny\gray $15$}}
\rput(19.5,2.7){{\tiny\gray $13$}}
\rput(19.5,3.7){{\tiny\gray $11$}}
\rput(19.5,4.7){{\tiny\gray $9$}}
\rput(19.5,5.7){{\tiny\gray $7$}}
\rput(19.5,6.7){{\tiny\gray $5$}}
\rput(19.5,7.7){{\tiny\gray $3$}}
\rput(19.5,8.7){{\tiny\gray $1$}}
\rput(19.5,9.7){{\tiny\gray $-1$}}
\psset{linewidth=0.07,linecolor=blue,fillstyle=solid,fillcolor=white}
\pscircle(1,0){0.175}
\pscircle(3,1){0.175}
\pscircle(5,2){0.175}
\pscircle(7,3){0.175}
\pscircle(9,4){0.175}
\pscircle(11,5){0.175}
\pscircle(13,6){0.175}
\pscircle(15,7){0.175}
\pscircle(17,8){0.175}
\pscircle(19,9){0.175}
\end{pspicture}

%% file: graphics/dyck-path.tex
\psset{unit=0.4cm}
\begin{pspicture}(-1.0,-7.5)(11.0,5.5)
\psset{linewidth=0.05,fillstyle=none,linecolor=lightgray}
\pspolygon[fillstyle=solid,fillcolor=lightgray](0,0)(3,3)(4,2)(6,4)(7,3)(9,1)(1,-7)(0,-6)(1,-5)(0,-4)(1,-3)(0,-2)(1,-1)
\psset{linewidth=0.175,fillstyle=none,linecolor=gray}
\psline(0,0)(3,3)(4,2)(6,4)(7,3)(10,0)
\psset{linewidth=0.02,fillstyle=none,linecolor=black}
\psline{->}(-1,0)(11,0)
\psline{->}(0,-7)(0,5)
\psset{linewidth=0.05,fillstyle=none,linecolor=black}
\psline{<->}(0.5,-0.5)(3.5,2.5)
\rput{45}(2,1.5){{\scriptsize\blue $3$}}
\psline{<->}(0.5,-2.5)(6.5,3.5)
\rput{45}(3,0.5){{\scriptsize\blue $6$}}
\psline{<->}(0.5,-4.5)(7.5,2.5)
\rput{45}(4,-0.5){{\scriptsize\blue $7$}}
\psline{<->}(0.5,-6.5)(8.5,1.5)
\rput{45}(5,-1.5){{\scriptsize\blue $8$}}
\end{pspicture}